\documentclass{amsart}
\usepackage{amssymb,latexsym}
\usepackage{amsthm}
\usepackage{amscd}
\usepackage[all]{xy}
\usepackage[mathscr]{eucal}
\usepackage{amsxtra}
\usepackage{stmaryrd}
\usepackage{mathabx}
\usepackage{calc}
\usepackage{tikz}
\tikzstyle{vertex}=[circle, draw, inner sep=0pt, minimum size=3pt]
\newcommand{\vertex}{\node[vertex][fill]}

\DeclareSymbolFont{AMSb}{U}{msb}{m}{n}
\DeclareMathSymbol{\N}{\mathbin}{AMSb}{"4E}
\DeclareMathSymbol{\Z}{\mathbin}{AMSb}{"5A}
\DeclareMathSymbol{\R}{\mathbin}{AMSb}{"52}
\DeclareMathSymbol{\Q}{\mathbin}{AMSb}{"51}
\DeclareMathSymbol{\C}{\mathbin}{AMSb}{"43}
\newcommand{\letbe}{\mathbin{:\!\raisebox{-.36pt}{=}}\,}

\DeclareMathOperator{\Hom}{Hom}

\theoremstyle{plain}
\newtheorem{lemma}{Lemma}[section]
\newtheorem{proposition}[lemma]{Proposition}
\newtheorem{theorem}[lemma]{Theorem}
\newtheorem{corollary}[lemma]{Corollary}
\newtheorem{conjecture}[lemma]{Conjecture}

\theoremstyle{definition}
\newtheorem{definition}[lemma]{Definition}
\newtheorem{example}[lemma]{Example}

\newtheorem{remark}[lemma]{Remark}

\theoremstyle{remark}

\numberwithin{equation}{section}

\begin{document}

\title{Torus Manifolds in Equivariant Complex Bordism}
\author{Alastair Darby}
\address{School of Mathematics\\
The University of Manchester\\
Oxford Road\\
Manchester\\
M13 9PL.}
\email{alastair.darby@manchester.ac.uk}
\date{\today}

\begin{abstract}
We restrict geometric tangential equivariant complex $T^n$-bordism to torus manifolds and provide a complete combinatorial description of the appropriate non-commutative ring. We discover, using equivariant $K$-theory characteristic numbers, that the information encoded in the oriented torus graph associated to a stably complex torus manifold completely describes its equivariant bordism class. We also consider the role of omnioriented quasitoric manifolds in this description.
\end{abstract}

\maketitle

\section{Introduction}

Bordism is a well known equivalence relation on compact manifolds of the same dimension, using the notion of boundary, whose importance was revealed by Ren\'e Thom in the 1950s when he showed that bordism groups could be computed using homotopy theory. Thom's work was concerned with what we now call \emph{unoriented} bordism, where two closed $n$-dimensional manifolds, $M$ and $N$, are \emph{bordant} if there is a compact $(n+1)$-manifold whose boundary is the disjoint union $M\sqcup N$. The bordism class of a manifold $[M]\in \Omega_n^O$ is completely determined by its \emph{Stiefel-Whitney characteristic numbers}. The set of bordism classes forms an abelian group $\Omega_n^O$ for which addition is induced by disjoint union and the identity is represented by the empty manifold. Cartesian product of manifolds induces a product structure and we obtain a graded ring $\Omega_*^O\letbe \oplus_{n\geq 0} \Omega_n^O$. Thom computes this to be isomorphic to a $\Z /2$-polynomial algebra by showing that it is isomorphic to the homotopy groups of the \emph{Thom spectrum} $MO$, defined by considering the \emph{Thom spaces} of the tautological bundles that classify real vector bundles. This is achieved by considering the \emph{Pontryagin-Thom construction}, which provides us with a map $\Omega_*^O\to MO_*$, and finding an inverse using \emph{transversality} arguments.

Independently, Milnor and Novikov gave a \emph{complex} variant of bordism using the Thom spectrum $MU$, defined by considering the Thom spaces of the tautological bundles that classify complex vector bundles. It was shown that $MU$ represents bordism classes of \emph{stably complex manifolds}, which admit a complex structure on the direct sum of their tangent bundle and a trivial bundle of sufficiently large rank. Complex bordism classes are completely determined by their \emph{Chern characteristic numbers}, and constitute a graded ring $\Omega_*^U$. Transversality, as in the unoriented case, provides us with an isomorphism of $\Omega_*^U$ with the homotopy groups of $MU$, which is known to be a polynomial algebra over $\Z$. Every spectrum provides us with a homology and cohomology theory and the complex Thom spectrum $MU$ is \emph{universal} among \emph{complex-orientable} cohomology theories.

In an attempt to further the study of transformation groups, Conner and Floyd \cite{CF64} define an \emph{equivariant} version of geometric bordism. This is, by analogy with the non-equivariant case, an equivalence relation on compact (stably complex) $G$-manifolds, for a compact Lie group $G$, which results in graded rings $\Omega_*^{O:G}$ and $\Omega_*^{U:G}$. An equivariant version of Thom spectra was given by tom Dieck~\cite{tD70} resulting in the \emph{$G$-spectra} $MO_G$ and $MU_G$. It is still possible to define maps $\Omega_*^{O:G}\to MO_*^G$ and $\Omega_*^{U:G}\to MU_*^G$ using the Pontryagin-Thom construction, but unfortunately, for non-trivial $G$, transversality does not extend to the equivariant case. Therefore, we do not obtain an isomorphism between the rings defined geometrically and the homotopy groups of the equivariant Thom spectra.

We will concentrate on equivariant complex bordism with respect to the torus $T^k\letbe (S^1)^k$, for some $k\geq 1$. This has been proven to be important thanks to the work of tom Dieck~\cite{tD70} and L\"offler~\cite{L73} who show that the Pontryagin-Thom map above is injective in this case. Also $MU_A$, for a compact abelian Lie group $A$, is known to be universal among $A$-equivariant complex-orientable theories; see~\cite{CGK02}.

Goresky, Kottwitz and MacPherson~\cite{GKM98} studied a particular class of almost complex manifolds with effective torus actions. They associated a labelled graph to each manifold using its fixed point data and showed that this combinatorial data was enough to calculate the equivariant cohomology of the manifold. Rosu~\cite{R03} was also able to give a description, using this combinatorial data, of the equivariant $K$-theory (with complex coefficients) of the manifold. Our aim is to show that this data is also what we need to determine the bordism class of the manifold although we will widen the class of manifolds, and consequently alter the labelled graphs, in question by looking at \emph{stably} complex manifolds with effective torus actions.

In this paper we will only be concerned with the case of half-dimensional torus actions, that is, even-dimensional manifolds $M^{2n}$ with effective $T^n$-actions. This will imply that the only fixed points that appear are isolated fixed points. The GKM-theory of these particular manifolds was studied in depth in~\cite{MMP07} where the graphs were given the name \emph{torus graphs}. We will show how we can find the equivariant versions of characteristic numbers from the torus graphs and classify those that can appear.

A special class of such manifolds are known as omnioriented quasitoric manifolds (see Section~\ref{OQTM}). They were first defined in~\cite{DJ91} as topological versions of toric varieties and are central objects in the study of the topology of torus actions. Their attraction is the ability to completely encode them in terms of combinatorial data. Buchstaber and Ray in~\cite{BR98} prove that every (non-equivariant) complex bordism class of dimension~$>2$ contains an omnioriented quasitoric manifold. This paper is the beginning of an attempt to to prove the equivariant version of this theorem, that is, that an omnioriented quasitoric manifold is contained in every equivariant complex bordism class with suitable dimensions.

There are real versions of quasitoric manifolds known as \emph{small covers}, where we have nice $(\Z/2)^n$-actions on manifolds $M^n$. L\"u and Tan in~\cite{LT13} show that a small cover is contained in every $(\Z/2)^n$-equivariant unoriented bordism class of dimension $n$. This paper could then be considered as an attempt at a complex version of their paper and we extend many of their techniques to this case.

We begin in Section~\ref{SCTM} by looking at manifolds $M^{2n}$ with an effective $T^n$-action. These are known as \emph{torus manifolds} and we are interested in those that admit a stably complex structure. We then consider \emph{torus graphs}, first discussed in~\cite{MMP07}, which are combinatorial objects associated to torus manifolds; oriented versions are associated to those torus manifolds with a stably complex structure. We assign to these an exterior polynomial, and define a boundary operator that completely characterises such polynomials.

In Section~\ref{ECB} we give an introduction to geometric equivariant complex bordism and the technique of restricting to fixed point data. We then consider the case of bordism classes that contain manifolds with only isolated fixed points and study the \emph{universal toric genus} of~\cite{BPR10} and the pullback squares of Hanke~\cite{Han05} to obtain a fixed point formula.

Section~\ref{KTHY} brings the previous material together to show that every exterior polynomial coming from an oriented torus graph is the fixed point data for some equivariant bordism class. This gives our restricted bordism ring a completely combinatorial structure and the exterior polynomials may be thought of as equivariant characteristic `numbers'.

In Section~\ref{OQTM} we give a brief introduction to omnioriented quasitoric manifolds and define a ring of \emph{quasitoric pairs} that completely encodes omnioriented quasitoric manifolds. We then consider whether this particular class of stably complex torus manifolds could generate all the bordism classes, give an affirmative answer in low dimensions and discuss problems that arise in higher dimensions.

\subsection{Preliminaries}

Let $G$ be a compact Lie group. For a good exposition on equivariant algebraic topology, including the definitions of $G$-bundle, $G$-homotopy etc., we refer the reader to the classical text~\cite{May96}. Suppose $\xi$ is a real $2n$-plane $G$-bundle.

\begin{definition}
A \emph{complex $G$-structure on $\xi$} is a $G$-map
\[
J\colon E(\xi)\longrightarrow E(\xi),
\]
satisfying $J^2=-1$. Two complex structures are \emph{equivalent} if they are $G$-homotopic through complex $G$-structures.
\end{definition}

Note that equivalence classes of complex $G$-structures are in one-to-one correspondence with $G$-homotopy classes of bundle maps $\xi\to \gamma_G^n$, where
\[
\gamma_G^n\colon EU_G(n)\longrightarrow BU_G(n)
\]
is the universal complex $n$-plane $G$-bundle which we regard as a real $2n$-plane $G$-bundle.

\begin{definition}\label{D:tsc}
A \emph{tangentially stably complex $G$-manifold} $M$ is a compact $G$-manifold with a complex $G$-structure on
\[
\tau(M)\oplus \R^k,
\]
for some $k$, where the trivial bundle $\R^k$ has trivial $G$-action. Two such structures are defined to be equivalent, if after stabilisation with further $\C$-summands (equipped with trivial $G$-actions on the fibres), the induced complex $G$-structures are equivalent.
\end{definition}

An important property of these objects that we will use is given in the following well-known proposition:

\begin{proposition}[See~\cite{Com96}]\label{P:tsc}
Suppose $M$ is a tangentially stably complex $G$-manifold. Then each component of the fixed point set $M^G$ is a tangentially stably complex $G$-manifold and its normal bundle in $M$ is a complex $G$-bundle.
\end{proposition}

\begin{proof}
Suppose we have a complex $G$-structure on $\tau(M)\oplus \R^k$, for some $k$. Observe that $\tau(M^G)=(\tau(M)|_{M^G})^G$ as real vector bundles. Then
\[
\zeta:=(\tau(M)\oplus \R^k)|_{M^G}=\tau(M)|_{M^G}\oplus \R^k=\tau(M^G)\oplus \nu(M^G,M) \oplus \R^k.
\]
Since $\zeta^G=\tau(M^G)\oplus \R^k$, we have that $(\zeta^G)^{\perp}=\nu(M^G,M)$ is a complex $G$-bundle.
\end{proof}

\begin{remark}
If we were to stabilise the tangent bundle with respect to \emph{all} complex $G$-representations, as opposed to just the trivial ones, then Proposition~\ref{P:tsc} would not hold; we would only be able to say that the normal bundles to fixed sets would be \emph{stably} complex. The same would be true if we were to define a theory using the stable normal bundle instead of the stable tangent bundle.
\end{remark}

We will use the following important fact about $G$-bundles over a trivial $G$-space:

\begin{proposition}[{\cite[Proposition 2.2]{Seg68}}]\label{P:segal}
Let $\xi$ be a complex $G$-bundle over a compact space $X$ that has trivial $G$-action. Then $\xi$ decomposes as the direct sum
\[
\xi\cong \bigoplus_{V} \xi_V,
\]
where $V$ runs over irreducible $G$-representations and $\xi_V\cong \tilde{\xi}\otimes V$ for some non-equivariant complex vector bundle $\tilde{\xi}$ over $X$.
\end{proposition}

If a tangentially stably complex $G$-manifold $M$ has an isolated fixed point $p$, then, by Proposition~\ref{P:tsc}, $M$ is even dimensional and the tangent space $T_pM$ is a complex $G$-representation. Since $M$ has a canonical orientation coming from its stably complex structure, this induces an orientation on $T_pM$. Therefore $T_pM$ has two orientations: one coming from its complex structure, and the other from the orientation of $M$.

We define the \emph{sign} of an isolated fixed point as in~\cite{Mas99,BPR07}:

\begin{definition}\label{D:sign}
For each isolated fixed point $p$ of a tangentially stably complex $G$-manifold, the \emph{sign of $p$} is given by
\[
\sigma(p)\letbe
\begin{cases}
+1, &\text{if the two orientations coincide};\\
-1, &\text{if the two orientations differ}.
\end{cases}
\]
\end{definition}

Note that $\sigma(p)$ agrees with the induced orientation on $p$ as a stably complex manifold in its own right and that for an almost complex manifold $\sigma(p)=+1$ for all isolated fixed points $p\in M$.

\subsubsection*{Acknowledgments}

This work comes from part of my PhD thesis~\cite{D13} and I would like to thank my doctoral supervisor Nigel Ray for all the help, comments and confidence that he has given throughout my time spent in Manchester.

I would also like to thank Neil Strickland for his many helpful comments on this project.

\section{Stably Complex Torus Manifolds}\label{SCTM}

In this section we study what we call \emph{stably complex torus manifolds}. Masuda studied these in~\cite{Mas99} in which he called them ``unitary toric manifolds''. We give them a different name to avoid confusion with compact non-singular toric varieties and quasitoric manifolds, which have been called toric manifolds in the literature. Much of this section may be considered as an extension of the ideas of L\"u and Tan~\cite[\S 3,4]{LT13} to the complex case.

\begin{definition}
A \emph{torus manifold} is a $2n$-dimensional smooth compact manifold $M$ with an effective smooth $T^n$-action whose fixed point set is non-empty.
\end{definition}

Note that for a manifold $M^{2n}$ with an effective $T^n$-action the fixed point set $M^{T^n}$ is necessarily finite. So for a torus manifold there are only isolated fixed points.

\begin{definition}
A \emph{stably complex torus manifold} $M^{2n}$ is a tangentially stably complex $T^n$-manifold whose $T^n$-action also satisfies the conditions of being a torus manifold, that is, the $T^n$-action is effective and has a non-empty fixed point set.
\end{definition}

Following the work in~\cite{Mas99} we define a \emph{characteristic submanifold} of a stably complex torus manifold to be a closed connected codimension 2 submanifold that is a fixed point set component of a certain circle subgroup of $T^n$ that contains at least one $T^n$-fixed point. Compactness ensures that for a stably complex torus manifold $M$ there are only a finite number of characteristic submanifolds. Let $M_i$ be the characteristic submanifolds of $M$, $\nu_i$ be their normal bundles in $M$ and let $T_i$ be the circle subgroup which fixes $M_i$ pointwise. For $p\in M^{T^n}$ we set
\[
I(p)\letbe \{ i\mid p\in M_i\}.
\]
Then
\begin{equation}\label{E:TpM}
T_pM\cong\bigoplus_{i\in I(p)}\nu_i|_p
\end{equation}
as complex $T^n$-modules. This shows that $I(p)$ consists of exactly $n$ elements. Let $V_i(p)\in \Hom(T^n,S^1)\cong \Z^n$ be the irreducible complex $T^n$-representation associated to each $\nu_i|_p$. Using the fact that the $T^n$-action is effective Masuda obtained the following result:

\begin{lemma}[{\cite[Lemma 1.3]{Mas99}}]\label{L:Mas}
The following conditions hold for each $p\in M^{T^n}$:
\begin{enumerate}
\item The set $\{ V_i(p)\mid i\in I(p)\}$ forms a basis of $\Hom(T^n,S^1)$.
\item Let $j\in I(p)$. Then $\text{Res}_{T_j}(V_i(p))\neq 0$ if and only if $j=i$, where $\text{Res}_{T_j}$ denotes the restriction map from $\Hom(T^n,S^1)$ to $\Hom(T_j,S^1)$.
\item If $p$ and $q$ are points in $M_i$, the $\text{Res}_{T_i}(V_i(p))=\text{Res}_{T_i}(V_i(q))$.
\end{enumerate}
\end{lemma}

We now consider combinatorial objects called torus graphs and look at the torus graphs that are obtained from stably complex torus manifolds. We then go on to define an exterior polynomial associated to each such torus graph and a boundary operator which characterises those exterior polynomials that arise in this fashion.

\subsection{Torus Graphs}

Torus graphs were first introduced in~\cite{MMP07} as combinatorial objects used to study simplicial posets and to calculate the equivariant cohomology of equivariantly formal torus manifolds by analogy with the way that the GKM graphs of~\cite{GKM98} and~\cite{GZ99} deal with GKM manifolds. L\"u and Tan study a mod 2 version of these torus graphs in~\cite{LT13}.

Suppose $\Gamma$ is an $n$-valent connected graph with $n\geq 1$. Let $\mathcal{V}(\Gamma)$ denote the set of vertices of $\Gamma$ and $\mathcal{E}(\Gamma)$ the set of oriented edges of $\Gamma$, that is, each edge of $\Gamma$ appears twice in $\mathcal{E}(\Gamma)$---once for each possible orientation. Let $i(e)$ and $t(e)$ be the initial and terminal vertices of an edge $e\in \mathcal{E}(\Gamma)$ respectively, and $\bar{e}$ be the edge $e$ with its opposite orientation. For $p\in \mathcal{V}(\Gamma)$ set
\[
\mathcal{E}(\Gamma)_p\letbe\{e\in \mathcal{E}(\Gamma)\mid i(e)=p \}.
\]

\begin{definition}\label{D:taf}
A \emph{torus axial function} is a map
\[
\alpha\colon \mathcal{E}(\Gamma)\longrightarrow \Hom(T^n,S^1)\cong \Z^n,
\]
satisfying the following conditions:
\begin{enumerate}
\item $\alpha(\bar{e})=\pm \alpha(e)$;
\item elements of $$\alpha(\mathcal{E}(\Gamma)_p)\letbe \{ \alpha(e)\in \Hom(T^n,S^1)\mid i(e)=p\}$$ form a basis of $\Z^n$;
\item $\alpha(\mathcal{E}(\Gamma)_{t(e)})\equiv \alpha (\mathcal{E}(\Gamma)_{i(e)})\ \text{mod}\ \alpha(e)$, for any $e\in \mathcal{E}(\Gamma)$. \label{3}
\end{enumerate}
\end{definition}

\begin{remark}
Note that a torus axial function is different from the axial function for GKM-graphs as defined in~\cite{GZ99}, which requires that $\alpha(\bar{e})=-\alpha(e)$ as well as that the elements of $\alpha(\mathcal{E}(\Gamma)_p)$ need only be pairwise linearly independent.
\end{remark}

\begin{definition}
A \emph{torus graph} is a pair $(\Gamma,\alpha)$ consisting of an $n$-valent graph $\Gamma$ with a torus axial function $\alpha$.
\end{definition}

\begin{example}[{\cite[p.463]{MMP07}}]
Let $M^{2n}$ be a torus manifold. Let $\Sigma_M$ denote the set of 2-dimensional submanifolds of $M$ each of which is fixed pointwise by a codimension one subtorus of $T^n$. Then every $S\in \Sigma_M$ is diffeomorphic to a sphere, contains exactly two $T^n$-fixed points, and is a connected component of the intersection of some $n-1$ characteristic submanifolds. Define a regular $n$-valent graph $\Gamma_M$ whose vertex set is $M^{T^n}$ and identify the edges of $\Gamma_M$ with the spheres from $\Sigma_M$. The summands in~\eqref{E:TpM} correspond to the oriented edges of $\Gamma_M$ having $p$ as the initial point. We assign $V_i(p)$ to the oriented edge corresponding to $\nu_i|_p$. This gives a function
\[
\alpha_M\colon \mathcal{E}(\Gamma_M)\longrightarrow \text{Hom}(T^n,S^1).
\]
It is not difficult to see that $\alpha_M$ satisfies the three conditions of being a torus axial function.
\end{example}

\begin{example}[{\cite[p.463]{MMP07}}]\label{E:sphere}
By suspending the coordinatewise $T^n$-action on $S^{2n-1}$, for $n>1$, we obtain the torus manifold $S^{2n}$. The torus graphs of these torus manifolds have only two vertices and the axial function assigns the basis elements $t_1,\dots,t_n\in \Hom (T^n,S^1)$ to the $n$ edges regardless of the orientation due to the condition given by Definition~\ref{D:taf}(3). We draw the graph below in the case that $n=3$.
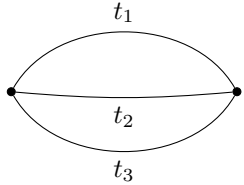
\[
\begin{tikzpicture}
\vertex (v) at (0,0){};
\vertex (w) at (3,0){};
\path
(v) edge[bend right=5] node[below]{$t_2$} (w)
(v) edge[bend left=60] node[above]{$t_1$} (w)
(v) edge[bend right=60] node[below]{$t_3$} (w)
;
\end{tikzpicture}
\]
Note that these are not GKM-graphs as in~\cite{GZ99} as the condition $\alpha(\bar{e})=-\alpha(e)$ is not satisfied.
\end{example}

\begin{definition}\label{D:o}
An \emph{orientation} of a torus graph $(\Gamma,\alpha)$ is an assignment
\[
\sigma\colon \mathcal{V}(\Gamma)\longrightarrow \{ \pm 1\},\]
satisfying
\[
\sigma (i(e))\alpha (e)=-\sigma (i(\bar{e}))\alpha (\bar{e}),\quad \text{for every}\ e\in \mathcal{E}(\Gamma).
\]
\end{definition}

\begin{example}[{\cite[p.476]{MMP07}}]\hfill
\begin{enumerate}
\item Not all torus graphs are orientable. Take the complete graph on four vertices $v_1,v_2,v_3,v_4$. Choose a basis $t_1,t_2,t_3\in \Hom (T^3,S^1)$ and define an axial function by setting
\[
\alpha(v_1v_2)=\alpha(v_3v_4)=t_1,\quad \alpha(v_1v_3)=\alpha(v_2v_4)=t_2,\quad \alpha(v_1v_4)=\alpha(v_2v_3)=t_3,
\]
and $\alpha(\bar{e})=\alpha(e)$ for any oriented edge $e$. A direct check shows that this torus graph is not orientable.
\item A $T^n$-invariant almost complex structure on a torus manifold $M^{2n}$ induces orientations on $M$ and on its characteristic submanifolds. The associated torus axial function satisfies $\alpha_M(\bar{e})=-\alpha_M(e)$ for any oriented edge $e$. In this case we may take $\sigma(p)=+1$ for every $p\in V(\Gamma_M)$.
\end{enumerate}
\end{example}

\begin{proposition}
The torus graph of a stably complex torus manifold is orientable.
\end{proposition}

\begin{proof}
Any stably complex torus manifold $M$ is a tangentially stably complex $T^n$-manifold so set $\sigma(p)$, for $p\in M^{T^n}=\mathcal{V}(\Gamma_M)$, to agree with the definition of the sign of an isolated fixed point of a tangentially stably complex $T^n$-manifold as in~Definition~\ref{D:sign}. The condition for the torus graph to be orientable is then satisfied.
\end{proof}

\begin{example}\label{E:sphere2}
The torus graph of the torus manifold $S^{2n}$, of Example~\ref{E:sphere}, can be oriented by giving opposite signs to the two fixed points.
\end{example}

\subsection{Torus Polynomials}

We begin by defining the \emph{free exterior algebra}, on a set $S$, over a commutative ring $R$. This can be thought of as the non-commutative analogue of a polynomial ring.

\begin{definition}
The \emph{free $R$-algebra} $R\langle S\rangle$, on $S$, is the free $R$-module with basis consisting of all finite words over $S$. This becomes an $R$-algebra by defining multiplication as concatenation of basis elements.
\end{definition}

Let $S=\{ X_i\mid i\in I\}$. Then it is clear that each element of $R\langle S\rangle$ can be uniquely written in the form:
\begin{equation}\label{eq:sum}
\sum_{i_1,\dots,i_k\in I} a_{i_1,\dots,i_k} X_{i_1}\cdots X_{i_k},
\end{equation}
where $a_{i_1,\dots,i_k}$ are elements of $R$, of which all but finitely many are non-zero. Unlike an actual polynomial ring the variables do not commute, i.e.\ $X_1X_2$ does not equal $X_2X_1$.

\begin{definition}
The \emph{free exterior $R$-algebra} $\Lambda(S)$ on $S$ is the quotient algebra
\[
\Lambda(S)\letbe R\langle S\rangle/I,
\]
where $I$ is the ideal generated by all words of the form $X_iX_i$ and expressions of the form $X_iX_j+X_jX_i$, for $X_i,X_j\in S$. The exterior product $\wedge$ of two elements is defined as $X_i\wedge X_j=X_iX_j\mod I$.
\end{definition}

The \emph{$k^{\text{th}}$ exterior power}, denoted $\Lambda^k(S)$, is the subspace of $\Lambda(S)$ spanned by elements of the form
\[
X_1\wedge X_2\wedge\dots\wedge X_k
\]
and we obtain a graded structure:
\[
\Lambda(S)=\Lambda^0(S)\oplus \Lambda^1(S)\oplus\cdots,
\]
where $(\Lambda^k(S))\wedge (\Lambda^p(S))\subset \Lambda^{k+p}(S)$.

We will always refer to each element in $\Lambda(S)$ in terms of a representative in $R\langle S\rangle$ and assume that it is written as a sum of basis elements as in~\eqref{eq:sum}.

\begin{definition}\label{D:J}
Let $J_n$ denote the set of non-trivial irreducible $T^n$-representations, that is, non-trivial elements of $\Hom (T^n,S^1)$.
\end{definition}

We have that $J_n\cong \Z^n\smallsetminus \{0\}$. Consider the free exterior algebra $\Lambda(J_n)$, over $\Z$.

\begin{definition}
We call an exterior polynomial in $\Lambda^n(J_n)$ \emph{faithful} if the indeterminates from each monomial form a basis of $\Z^n$.
\end{definition}

Suppose $(\Gamma,\alpha,\sigma)$ is an oriented torus graph. For a vertex $p$ we order the basis elements
\[
\alpha(\mathcal{E}_p)=\{ \alpha(e_1),\dots,\alpha(e_n)\}
\]
so that
\begin{equation}\label{E:tp}
\det [\alpha(e_1)\cdots\alpha(e_n)]=\sigma(p),
\end{equation}
where $[\alpha(e_1)\cdots\alpha(e_n)]$ is the integral $(n\times n)$-matrix formed by taking the $i^{\text{th}}$ column to be the vector $\alpha(e_i)\in\Z^n$. This defines a faithful exterior monomial
\[
\mu_p=\alpha(e_1)\wedge\dots\wedge \alpha(e_n)\in \Lambda^n(J_n),
\]
for every vertex $p$ of $\Gamma$. Note that this monomial is independent of the ordering chosen above.

\begin{remark}
In the case $n=1$ there is only one ordering of $\alpha(\mathcal{E}_p)$ since it contains a single element $\alpha(e)$. In this case we define
\[ \mu_p\letbe
\begin{cases}
+\alpha(e), &\text{if}\ \det[\alpha(e)]=\sigma(p);\\
-\alpha(e), &\text{if}\ \det[\alpha(e)]\ne\sigma(p).
\end{cases} \]
\end{remark}

\begin{definition}
The \emph{torus polynomial} of an oriented torus graph $(\Gamma, \alpha)$ is defined to be the faithful exterior polynomial
\[
g{(\Gamma,\alpha)}\letbe \sum_{p\in \mathcal{V}(\Gamma)} \mu_p\in \Lambda^n(J_n).
\]
\end{definition}

\begin{example}
Combining Examples~\ref{E:sphere} and~\ref{E:sphere2} we see that the spheres $S^{2n}$ have zero torus polynomial because the two monomials coming from the two vertices cancel as they have opposite sign.
\end{example}

We define an equivalence relation on the set of oriented torus graphs by saying that for two torus graphs $(\Gamma_1,\alpha_1)$ and $(\Gamma_2,\alpha_2)$ are equivalent if and only if $g{(\Gamma_1,\alpha_1)}=g{(\Gamma_2,\alpha_2)}$, and denote the set of equivalence classes by $G(n)$. On $G(n)$ we define addition to be disjoint union and so form an abelian group where the zero element is given by the class of any torus graph with zero torus polynomial, or equivalently, the empty graph.

We then have a monomorphism of abelian groups
\begin{equation}\label{e:g}
g\colon G(n)\longrightarrow \Lambda^n(J_n)
\end{equation}
given by taking the torus polynomial of an oriented torus graph.

\begin{definition}
An oriented torus graph $(\Gamma,\alpha)$ with non-zero torus polynomial is said to be \emph{prime} if the number of monomials in $g{(\Gamma,\alpha)}$ is equal to the number of vertices of $\Gamma$.
\end{definition}

\begin{lemma}
Each non-zero class of $G(n)$ contains a prime torus graph as its representative.
\end{lemma}

\begin{proof}
Let $(\Gamma,\alpha,\sigma)$ be an oriented torus graph with $g{(\Gamma,\alpha)}\neq 0$ and assume that $(\Gamma,\alpha)$ is not prime. Then there must be two vertices $p$ and $q$ such that $\mu_p+\mu_q=0$. Notice that this can happen only if $\sigma(p)\ne\sigma(q)$. We can now form a new torus graph by removing the two vertices $p$ and $q$ and gluing the $n$ edges $\mathcal{E}_p=\{e_1^p,\dots,e_n^p\}$ with $p$ removed to the $n$ edges $\mathcal{E}_q=\{e_1^q,\dots,e_n^q\}$ with $q$ removed in such a way that $e_i^p$ will be glued to $e_j^q$ whenever $\alpha(e_i^p)=\alpha(e_j^q)$. The resulting graph $(\Gamma',\alpha')$ will be a torus graph with two fewer vertices than $(\Gamma,\alpha)$, such that $(\Gamma',\alpha')\sim (\Gamma,\alpha)$. Since $\Gamma$ is finite this procedure will terminate and we will obtain a prime torus graph.
\end{proof}

\begin{remark}
If we have two oriented torus graphs $(\Gamma_1,\alpha_1)$ and $(\Gamma_2,\alpha_2)$ with vertices $p_1\in \mathcal{V}(\Gamma_1)$ and $p_2\in \mathcal{V}(\Gamma_2)$ such that $\mu_{p_1}+\mu_{p_2}=0$, then we can form the \emph{connected sum} $(\Gamma_1,\alpha_1)\#_{p_1,p_2}(\Gamma_2,\alpha_2)$ in a similar fashion to the proof above and we obtain the following formula:
\[
g((\Gamma_1,\alpha_1)\#_{p_1,p_2}(\Gamma_2,\alpha_2))=g(\Gamma_1,\alpha_1)+g(\Gamma_2,\alpha_2)\in \Lambda^n(J_n).
\]
\end{remark}

\subsection{A Boundary Operator}

We can see that both $\Hom(T^n,S^1)$ and $\Hom(S^1,T^n)$ are isomorphic to $\Z^n$ and are dual by the pairing
\begin{equation}\label{E:duality}
\langle \cdot\, , \cdot\rangle\, \colon \Hom(S^1,T^n)\times \Hom(T^n,S^1)\longrightarrow \Hom(S^1,S^1)\cong \Z,
\end{equation}
defined by composition of homomorphisms, that is, $\langle\, \rho,\xi\, \rangle=\xi\circ\rho$.

Similarly to Definition~\ref{D:J}, we define $J_n^*$ to be the set of non-trivial elements of $\Hom(S^1,T^n)\cong \Z^n$. For each faithful exterior polynomial $h\in \Lambda^n(J_n)$ we can obtain a \emph{dual polynomial} $h^*\in \Lambda^n(J_n^*)$. This is defined by taking each monomial of $h$, which forms a basis of $\Hom(T^n,S^1)\cong\Z^n$, and considering the \emph{dual basis}, elements of which belong in $\Hom(S^1,T^n)\cong\Z^n$, which gives us a faithful exterior monomial in $\Lambda^n(J_n^*)$. In terms of matrices, if $A$ is the invertible matrix associated to a monomial in $h$, then the corresponding monomial in $h^*$ will be associated to the transpose of the inverse of $A$ and we set
\begin{equation}\label{matrixdual}
A^*\letbe(A^{-1})^T.
\end{equation}
Notice that we have the relation: $(AB)^*=A^*B^*$.

We now define a chain complex $(\Lambda^k (J_n^*),d_k)$ as follows: \\for each monomial $s_1\wedge\dots\wedge s_k\in \Lambda^k (J_n^*)$, with all $s_i\in J_n^*$,
\[ d_k(s_1\wedge\dots\wedge s_k)\letbe
\begin{cases}
\sum_{i=1}^k (-1)^{i+1} s_1\wedge\dots\wedge s_{i-1}\wedge \widehat{s}_i\wedge s_{i+1}\wedge\dots\wedge s_k, &\text{if $k>1$;}\\
1 &\text{if $k=1$.}
\end{cases} \]
and $d_0(1)=0$, where $\widehat{s}_i$ denotes that $s_i$ is deleted. It is easy to see that $d^2=0$.

\begin{lemma}\label{L:exact}
The sequence
\[
\cdots \to \Lambda^{i+1}(J_n^*) \xrightarrow{d_{i+1}} \Lambda^{i}(J_n^*) \xrightarrow[\phantom{d_{i+1}}]{d_{i}} \Lambda^{i-1}(J_n^*) \xrightarrow{d_{i-1}} \cdots \xrightarrow[\phantom{d_{i+1}}]{d_1} \Lambda^{0}(J_n^*)\xrightarrow[\phantom{d_{i+1}}]{d_0} 0
\]
is exact.
\end{lemma}

\begin{proof}
We only need to show that $\text{Ker}\ d_i\subseteq \text{Im}\ d_{i+1}$, for $i\geq 0$. Let $h\in \text{Ker}\ d_i$ and take $H=t\wedge h$ for some $t\in J_n^*$ such that $t$ is not an indeterminate of any monomial of $h$. Then $d_{i+1}(H)=h-t\wedge d_i(h)=h$.
\end{proof}

\begin{remark}
As L\"u and Tan in~\cite[Remark 7]{LT13} explain, although it is easy to see how we could define a similar boundary operator $d'$ on $\Lambda(J_n)$, for a faithful polynomial $h\in \Lambda^n(J_n)$, if $d(h^*)=0$, then generally $d'(h)$ does not equal zero.
\end{remark}

\begin{theorem}\label{T:tgraph}
Let $h\in \Lambda^n(J_n)$ be a faithful polynomial. Then $h=g(\Gamma,\alpha)$ is the torus polynomial of an oriented torus graph if and only if $d(h^*)=0$.
\end{theorem}

\begin{proof}
$(\Rightarrow)$ Suppose $h=g{(\Gamma,\alpha)}$ for some oriented torus graph $(\Gamma,\alpha)$. We write
\[
h=\sum_{p\in\mathcal{V}}\mu_p\quad \text{and}\quad h^*=\sum_{p\in\mathcal{V}}\mu_p^*,
\]
where $\mu_p^*$ is the \emph{dual monomial} of $\mu_p$, and each monomial can be written as
\[
\mu_p=\alpha(e_1)\wedge\dots\wedge\alpha(e_n)\quad \text{and}\quad \mu_p^*=\alpha^*(e_1)\wedge\dots\wedge \alpha^*(e_n),
\]
where $\alpha^*(e_i)\in J_n^*$, for $1\leq i\leq n$.

Then for an edge $e_i\in \mathcal{E}_p$ define
\[
(\mu_p^*)_{e_i}=(-1)^{i+1}\alpha^*(e_1)\wedge\dots\wedge \widehat{\alpha^*(e_i)}\wedge\dots\wedge \alpha^*(e_n)\in \Lambda^{n-1}(J_n^*),
\]
for $1\leq i\leq n$. Notice that
\begin{equation}\label{E:sumsum}
d(h^*)=\sum_{p\in \mathcal{V}} \sum_{e\in \mathcal{E}_p}(\mu_p^*)_e.
\end{equation}
We will show that for an edge $e$ with $i(e)=p$ and $t(e)=q$ we have
\begin{equation}
(\mu_p^*)_e+(\mu_q^*)_{\bar{e}}=0,
\end{equation}
which will imply that $d(h^*)=0$ by~\eqref{E:sumsum}.

For $n\geq 2$ we can always write
\[
\mu_p=\alpha(e)\wedge \alpha(e_2)\wedge\dots\wedge \alpha(e_n)\quad\text{and}\quad -\mu_q=\alpha(\bar{e})\wedge \alpha(e_2')\wedge\dots\wedge \alpha(e_n').
\]
We know that $\alpha(\bar{e})=\pm\alpha(e)$ depending on the sign of $p$ and $q$ by Definition~\ref{D:o}. We also have that $\alpha(e_i')=\alpha(e_{j_i})+k_i\alpha(e)$, for some $k_i\in \Z$, $2\leq i\leq n$, by Definition~\ref{D:taf}(3). So, letting $\{ \mathbf{e_1},\dots,\mathbf{e_n}\}$ denote the standard basis of $\Z^n$, we can write
\[
[\alpha(\bar{e})\ \alpha(e_2')\cdots \alpha(e_n')]=[\alpha(e)\ \alpha(e_2)\cdots \alpha(e_n)][\pm\mathbf{e_1}\ k_2\mathbf{e_1}+\mathbf{e_{j_2}}\cdots k_n\mathbf{e_1}+\mathbf{e_{j_n}}],
\]
as $(n\times n)$-integral matrices and by taking determinants and using Definition~\ref{D:o} we have that
\[
\det [\pm\mathbf{e_1}\ k_2\mathbf{e_1}+\mathbf{e_{j_2}}\cdots k_n\mathbf{e_1}+\mathbf{e_{j_n}}]=\pm1,
\]
since $\sigma(p)=\mp \sigma(q)$.
Taking duals (see~\eqref{matrixdual}) and, without loss of generality, we have,
\begin{align*}
[\alpha^*(\bar{e})\ \alpha^*(e_2')\cdots \alpha^*(e_n')]&=[\alpha^*(e)\ \alpha^*(e_2)\cdots \alpha^*(e_n)][\pm\mathbf{e_1}\ k_2\mathbf{e_1}+\mathbf{e_{2}}\cdots k_n\mathbf{e_1}+\mathbf{e_{n}}]^*\\
&=[\alpha^*(e)\ \alpha^*(e_2)\cdots \alpha^*(e_n)][\mathbf{e\ e_2\cdots e_n}],
\end{align*}
where $\mathbf{e}=\pm\mathbf{e_1}\mp k_2\mathbf{e_2}\mp\dots\mp k_n\mathbf{e_n}$. Therefore
\[
-\mu_q^*=\alpha^*(\bar{e})\wedge \alpha^*(e_2')\wedge\dots\wedge \alpha^*(e_n')=\alpha^*(\bar{e})\wedge \alpha^*(e_2)\wedge\dots\wedge \alpha^*(e_n).
\]
This implies that $-(\mu_q^*)_{\bar{e}}=(\mu_p^*)_e$, which proves the statement in this case.

In the special case that $n=1$, a direct check shows that the result holds.

$(\Leftarrow)$ We write
\[
h=\sum_{i=1}^m t_i=\sum_{i=1}^m t_{i,1}\wedge\dots\wedge t_{i,n}\quad \text{and}\quad h^*=\sum_{i=1}^m s_i=\sum_{i=1}^m s_{i,1}\wedge\dots\wedge s_{i,n},
\]
where $t_i^*=s_i$. Define
\[
[s_i]_j=(-1)^{j+1}s_{i,1}\wedge\dots\wedge \widehat{s_{i,j}}\wedge\dots\wedge s_{i,n}\in \Lambda^{n-1}(J_n^*),
\]
for $1\leq i\leq m$ and $1\leq j\leq n$. Notice that we can now write
\[
d(h^*)=\sum_{i=1}^m \sum_{j=1}^n [s_i]_j=0.
\]
We now construct an oriented torus graph. Take $m$ points $p_1,\dots,p_m$ as vertices and label them by $s_1,\dots,s_m$ respectively. For each vertex $p_i$ make a \emph{bouquet} around $p_i$ of $n$ segments and label them by $[s_i]_1,\dots,[s_i]_n$. Since $h^*$ is faithful and $d(h^*)=0$ we have that
\begin{equation}\label{E:2}
[s_i]_j+[s_k]_l=0,
\end{equation}
for some $k\ne i$, with $1\leq j,l\leq n$. We now join these segments to obtain an $n$-valent graph with labelled oriented edges. We relabel these edges in the following fashion:
\[
[s_i]_j\longmapsto t_{i,j}\in \Hom(T^n,S^1),
\]
to define a map $\alpha\colon \mathcal{E}(\Gamma)\to \Hom(T^n,S^1)$.
We now check that the conditions of being a torus graph are satisfied. It is obvious that $\alpha(\mathcal{E}_{p_i})$ is a basis for each vertex $p_i$ as $h$ is faithful. Suppose we have an edge $e$ such that $i(e)=p_i$ and $t(e)=p_k$. This implies that $s_i$ and $s_k$ agree on $(n-1)$-elements and satisfy~\eqref{E:2} for some $j,l$. Without loss of generality, for $n\geq 2$, we can write
\[
s_i=\phi\wedge \eta_2\wedge\dots\wedge\eta_n\quad \text{and}\quad -s_k=\psi\wedge \eta_2\wedge\dots\wedge\eta_n,
\]
Since $s_i$ is faithful we can write
\[
\psi=c_1\phi+c_2\eta_2+\dots+c_n\eta_n,
\]
for some $c_1,\dots,c_n\in \Z$. By setting $\mathbf{e}=(c_1,\dots,c_n)$ we can write
\[
[\psi\ \eta_2\cdots \eta_n]=[\phi\ \eta_2\cdots \eta_n][\mathbf{e}\ \mathbf{e_2}\cdots \mathbf{e_n}].
\]
Note that $c_1=\pm 1$. So
\begin{align*}
[\psi\ \eta_2\cdots \eta_n]^*&=[\phi\ \eta_2\cdots \eta_n]^*[\mathbf{e}\ \mathbf{e_2}\cdots \mathbf{e_n}]^*\\
&=[\phi\ \eta_2\cdots \eta_n]^*[c_1\mathbf{e_1}\ k_2\mathbf{e_1}+\mathbf{e_2}\cdots k_n\mathbf{e_1}+\mathbf{e_n}],
\end{align*}
where $k_j=-c_1c_j$. By writing
\[
[\phi\ \eta_2\cdots \eta_n]^*=[\alpha(e)\ \alpha(e_2)\cdots \alpha(e_n)],
\]
we see that
\begin{align*}
[\psi\ \eta_2\cdots \eta_n]^*&=[c_1\alpha(e)\ k_2\alpha(e)+\alpha(e_2)\cdots k_n\alpha(e)+\alpha(e_n)]\\
&=[\alpha(\bar{e})\ \alpha(e_2')\cdots \alpha(e_n')].
\end{align*}
It is easy to see that the other conditions of being a torus graph follows from this observation. Finally we need to check the orientation condition as stated in Definition~\ref{D:o}. We set $\sigma(p_i)=\det [s_i]$. We have that
\[
\sigma(p_k)=-\det [\psi\ \eta_2\cdots \eta_n]=-c_1\det [\phi\ \eta_2\cdots \eta_n]=-c_1\sigma(p_i).
\]
Therefore, $\sigma(p_i)=\sigma(p_k)$ if and only if $c_1=-1$, which implies that $\alpha(\bar{e})=-\alpha(e)$. Similarly, $\sigma(p_i)=-\sigma(p_k)$ if and only if $c_1=+1$, which implies that $\alpha(\bar{e})=\alpha(e)$. We then get that $h=g(\Gamma,\alpha)$.

It is easy to check that this is true for $n=1$ as well.
\end{proof}

\begin{definition}\label{D:K}
Let $K_n$ denote the abelian group of all faithful exterior polynomials $h\in \Lambda^n(J_n)$ such that $d(h^*)=0$.
\end{definition}

\begin{corollary}
We have the following isomorphism of abelian groups
\[
G(n)\cong K_n.
\]
\end{corollary}

\begin{proof}
By Theorem~\ref{T:tgraph} the image of the homomorphism $g$, see~\eqref{e:g}, is $K_n$. By construction the kernel of $g$ is trivial and so we have the required isomorphism.
\end{proof}

\section{Equivariant Complex Bordism}\label{ECB}

We introduce a bordism theory related to tangentially stably complex $G$-manifolds, for a compact Lie group $G$, and consider the normal data around fixed point sets. We restrict our attention to classes containing tangentially stably complex $T^n$-manifolds that have a finite fixed point set, discuss the \emph{universal toric genus} of~\cite{BPR10} and extend the pullback square of Hanke~\cite{Han05}.

\begin{definition}
For a $G$-space $X$ the \emph{bordism groups} of tangentially stably complex $G$-manifolds, denoted by $\Omega_m^{U:G}(X)$, are $G$-bordism classes of singular tangentially stably complex $G$-manifolds $M^m\to X$.
\end{definition}

The coefficients of this theory
\[
\Omega_*^{U:G}\letbe \Omega_*^{U:G}(\text{pt.})\letbe \bigoplus_{m\geq 0}\Omega_m^{U:G}(\text{pt.})
\]
are equipped with a ring structure induced by the diagonal $G$-action on the cartesian product of two $G$-manifolds.

\subsection{Restriction to Fixed Point Data}

Restricting to fixed point data allows us to use non-equivariant techniques in equivariant topology. We now consider this approach for equivariant bordism which was originally studied by Conner and Floyd~\cite{CF64} with tom Dieck~\cite{tD70}, Sinha~\cite{Sin01} and Hanke~\cite{Han05} providing further work. We use Hattori's~\cite{Hat74} notation and viewpoint.

\begin{definition}
Let $\mathcal{F}_*^G$ denote the bordism ring of pairs $(X,\nu)$, where $X$ is a stably complex manifold and $\nu$ is a complex $G$-bundle over $X$ without trivial irreducible factors in the fibres.
\end{definition}

By restricting to normal data around the fixed point set we obtain a homomorphism
\[
\varphi_{\Omega}\colon \Omega_*^{U:G}\longrightarrow \mathcal{F}_*^G
\]
given by
\begin{align*}
\varphi_{\Omega}[M]&=[M^G,\nu({M^G},M)]\\
&=\sum_{\substack{F\subset M^G\\ F\, \text{connected}}} [F,\nu(F,M)].
\end{align*}

When working with the torus it turns out this map is extremely important as:

\begin{theorem}[\cite{HO72}]
If $G=T^n$, then $\varphi_{\Omega}$ is a monomorphism.
\end{theorem}

We will now focus all out attention on the case when $G=T^n$.

\begin{definition}
Let $\mathcal{Z}_*^{U:T^n}$ denote the subring of $\Omega_*^{U:T^n}$ given by elements that can be represented by a tangentially stably complex $T^n$-manifold where the fixed point set is finite.
\end{definition}

We will now study $\varphi_{\Omega}$ when restricted to $\mathcal{Z}_*^{U:T^n}$. Recall that $J_n$ denotes the set of non-trivial irreducible $T^n$-representations. Then for each fixed point $p\in M^{T^n}$ we may write
\begin{equation}
\nu(\{p\},M)\cong T_p M=V_1(p)\oplus\dots\oplus V_m(p),
\end{equation}
where the dimension of $M$ is $2m$ and $V_1(p),\dots,V_m(p)\in J_n$. Then $\varphi_{\Omega}$ restricted to $\mathcal{Z}_*^{U:T^n}$, which we call $\varphi_Z$, can be written as

\begin{align*}
\varphi_Z\colon \mathcal{Z}_*^{U:T^n}&\longrightarrow \Z[J_n]\\
[M]&\longmapsto \sum_{p\in M^{T^n}} \sigma(p) \prod_{i=1}^m V_i(p),
\end{align*}
where the sign of $p$, $\sigma(p)$, is as in Definition~\ref{D:sign}.

\begin{proposition}\label{P:Z}
The commutative diagram
\[ \xymatrix{
\mathcal{Z}_*^{U:T^n} \ar[r]^-{\iota} \ar[d]_-{\varphi_Z} &\Omega_*^{U:T^n} \ar[d]_-{\varphi_{\Omega}}\\
\Z [J_n] \ar[r]^-{\bar{\iota}} &\mathcal{F}_*^{T^n}
} \]
has all maps injective and is a pullback square.
\end{proposition}

\begin{proof}
We need to check that given an element $x\in \Omega_n^{U:T^n}$ such that $\varphi_{\Omega}(x)\subset \bar{\iota}(\Z[J_n])$, we can represent $x$ by a tangentially stably complex $T^n$-manifold whose fixed point set contains no components of positive dimension.

Suppose that $M^m$ is a tangentially stably complex $T^n$-manifold, that represents $x$, such that $M^{T^n}$ contains a connected component $F$ of positive dimension. We may write the normal bundle $\nu(F,M)$, of complex dimension $d$ say, as
\begin{equation}\label{e:nu}
\nu(F,M)=(E_1\otimes V_1)\oplus\dots\oplus (E_j\otimes V_j)
\end{equation}
due to Proposition~\ref{P:segal}. Since $\varphi_{\Omega}(x)\subset \bar{\iota}(\Z[J_n])$, this implies that $[F]=0\in \Omega_{m-2d}^U$. Therefore, there exists a  stably complex manifold $W$ (with trivial $G$-action) such that $\partial W=F$ and there is a $T^n$-bundle $\xi$ over $W$ that can be written as
\[
\xi=(\widetilde{E}_1\otimes V_1)\oplus\dots\oplus (\widetilde{E}_j\otimes V_j),
\]
such that $\widetilde{E}_i|_{F}\cong E_i$, for $1\leq i\leq j$.

We consider the total space $D(\xi)$ of the disc bundle as a manifold with boundary and write its tangent bundle as
\[
\tau(D(\xi))\cong p^*\tau(W)\oplus p^*\xi,
\]
where $p\colon D(\xi)\to W$ is the projection map. We therefore have a tangentially stably complex $T^n$-structure on $D(\xi)$ coming from the complex $T^n$-structure of $\xi$ and the fact that $W$ is a stably complex $T^n$-manifold.

We similarly have that, for a tubular neighbourhood $N(F)\cong D(\nu(F,M))$ of $F$ in $M$,
\[
\tau(M)|_{N(F)}\cong p^*\tau(F)\oplus p^*\nu(F,M)
\]
for the restriction of $p$ to $F$.

Therefore, we may identify ${D(\nu(F,M))}\subset \partial D(\xi)$ with ${N(F)}\subset M\times \{ 1\}\subset M\times I$ and using standard equivariant gluing and smoothing techniques, see~\cite{K07}, obtain a tangentially stably complex $T^n$-manifold whose boundary is $(M\times \{0\})\sqcup M'$. Thus $M'$ also represents $x$ in $\Omega_m^{U:T^n}$ and its fixed point set has one less component of positive dimension, that is
\[
(M')^{T^n}=M^{T^n}\smallsetminus \{F\}.
\]

If there are other connected components of positive dimension in $(M')^{T^n}$ we repeat this procedure until there are no components of positive dimension and obtain a representative of $x$ of the required type.
\end{proof}

\subsection{Universal Toric Genus}

We discuss a natural multiplicative transformation
\[
\Phi\colon \Omega_*^{U:T^n}\longrightarrow MU^*(BT^n_+),
\]
where the image is the non-equivariant complex cobordism of $BT^n$. The homomorphism $\Phi$ was first introduced by tom Dieck in~\cite{tD70}, studied further in~\cite{L73,Hat74,Sin01,Han05} and most recently in~\cite{BPR10}, where the results of Toric Topology were brought to bear. It was in~\cite[\S 2]{BPR10} that it was called the \emph{universal toric genus} in its relation to equivariant extensions of Hirzebruch genera. It can be defined as follows: let $\pi_M\colon M^m\to *$ be the unique projection and consider applying the Borel construction to obtain the smooth fibre bundle
\[
1\times_{T^n} \pi_M\colon ET^n\times_{T^n} M\longrightarrow BT^n.
\]
Then by the work of Quillen~\cite{Q71} there exists a Gysin homomorphism that gives us a map
\[
(1\times_{T^n} \pi_M)_*\colon MU^*((ET^n\times_{T^n} M)_+)\longrightarrow MU^{*-m}(BT^n_+),
\]
and we define the universal toric genus as
\[
\Phi [M^m]\letbe (1\times_{T^n} \pi_M)_*1\in MU^{-m}(BT^n_+).
\]

\begin{remark}
The universal toric genus factors through \emph{homotopical equivariant cobordism} $MU_*^{T^n}$, first defined in~\cite{tD70}, via an equivariant version of the Pontryagin-Thom construction. For more details see~\cite[\S 2]{BPR10} and~\cite[\S 3]{D13}.
\end{remark}

Let $S\subset MU^*(BT^n_+)$ be the multiplicative set generated by Euler classes $e(V)\in MU^{2}(BT^n_+)$ of bundles
\[
ET^n\times_{T^n} V\longrightarrow BT^n,
\]
for $V\in J_n$. We then obtain the following commutative diagram
\begin{equation}\label{e:sq}
\begin{aligned}
\xymatrix{
\Omega_*^{U:T^n} \ar[r]^-{\Phi} \ar[d]_-{\varphi_{\Omega}} &MU^*(BT^n_+) \ar[d]_-{\varphi_{MU}}\\
\mathcal{F}_*^{T^n} \ar[r]^-{S^{-1}\Phi} &S^{-1}MU^*(BT^n_+)
}
\end{aligned}
\end{equation}
where $S^{-1}MU^*(BT^n_+)$ is the localisation, $\varphi_{MU}$ the canonical homomorphism and
\[
S^{-1}\Phi [X,\nu]\letbe \frac{(1\times_{T^n} \pi_X)_*1}{e(V_1)^{|E_1|}\cdots e(V_j)^{|E_j|}},
\]
where we write $\nu=(E_1\otimes V_1)\oplus\cdots\oplus (E_j\otimes V_j)$ as in~\eqref{e:nu}. Note that $\varphi_{MU}$ is injective as $S$ does not contain any zero divisors. We now give the following result of Hanke:

\begin{theorem}[{\cite[Theorem 2]{Han05}}]\label{T:Han}
The commutative square~\eqref{e:sq} has all maps injective and is a pullback square of $MU_*$-algebras.
\end{theorem}

By considering the restriction of $\Omega_*^{U:T^n}$ to the subring $\mathcal{Z}_*^{U:T^n}$ we obtain the commutative square
\[ \xymatrix{
\mathcal{Z}_*^{U:T^n} \ar[rr]^-{\Phi\circ\iota} \ar[d]_-{\varphi_{Z}} & &MU^*(BT^n_+) \ar[d]_-{\varphi_{MU}}\\
\Z[J_n] \ar[rr]^-{S^{-1}\Phi\circ\bar{\iota}} & &S^{-1}MU^*(BT^n_+)
} \]
which is a pullback square by Proposition~\ref{P:Z} and Theorem~\ref{T:Han}. This then gives us the following fixed point formula:
\begin{equation}\label{e:utg}
(\Phi\circ\iota)[M^{2m}]=\sum_{p\in M^{T^n}}\sigma(p)\prod_{i=1}^m \frac{1}{e(V_i(p))}\in MU^*(BT^n_+).
\end{equation}

\section{Equivariant $K$-Theory Characteristic Numbers}\label{KTHY}

In this section we prove that there is an isomorphism between $\mathcal{Z}_{2n}^{U:T^n}$ and $K_n$, the abelian group of faithful exterior polynomials $h\in \Lambda^n(J_n)$ such that $d(h^*)=0$. This uses equivariant $K$-theory characteristic numbers and what we have previously discussed.

We follow the introduction to this subject given by Hattori in~\cite{Hat74}. We have a product structure on the set of isomorphism classes of irreducible $T^n$-representations, given by tensor product, and we can identify $K^*(BT^n_+)$ with the ring of all (possibly infinite) sums $\sum m_kV_k$, where $m_k\in \Z$ and $V_k\in \Hom(T^n,S^1)$.

\begin{definition}
Let $\mathcal{K}$ be the kernel of the augmentation
\begin{align*}
d\colon K^*(BT^n_+)&\longrightarrow \Z\\
\sum m_kV_k&\longmapsto \sum m_k.
\end{align*}
\end{definition}

Let $\mathbf{t}=(t_1,t_2,\dots)$ be a sequence of indeterminates and $V$ a $T^n$-representation. Define
\[
\gamma_t(V-\dim V)\letbe \prod_{i=1}^{\dim V}(1+t_1(V_i-1)+t_2(V_i-1)^2+\cdots)
\]
in $K^*(BT^n_+)\llbracket \mathbf{t} \rrbracket$, where $V=\bigoplus_{i=1}^{\dim V}V_i$. Naturally, $\gamma_t$ extends to a map
\[
\gamma_t\colon \mathcal{K}\longrightarrow K^*(BT^n_+)\llbracket \mathbf{t} \rrbracket
\]
such that
\[
\gamma_t(x+y)=\gamma_t(x)\gamma_t(y).
\]

Let $S\subset K^*(BT^n_+)$ denote the multiplicative subset generated by the $K$-theory Euler classes
\[
\lambda_{-1}(V)=\sum_{i\geq 0} (-1)^i \lambda^i(V)
\]
of the bundles $ET^n\times_{T^n} V\to BT^n$, for $V\in J_n$.

We now extend the commutative square~\eqref{e:sq} to the right using a Boardman homomorphism. That is, we have a natural transformation of multiplicative cohomology theories
\[
B\colon MU^*(BT^n_+)\longrightarrow K^*(BT^n_+)\llbracket \mathbf{t} \rrbracket
\]
uniquely characterised by the formula
\begin{equation}\label{e:Boardman}
B(e(V))=\lambda_{-1}(V)(\gamma_t(V-1))^{-1}
\end{equation}
for $V$ a non-trivial irreducible $T^n$-representation and $e(V)\in MU^2(BT^n_+)$ the Euler class of the bundle $ET^n\times_{T^n} V\to BT^n$. Hattori proves the following:

\begin{theorem}[\cite{Hat74}]
The commutative square
\[ \xymatrix{
\Omega_*^{U:T^n} \ar[r]^-{B\circ \Phi} \ar[d]_-{\varphi_{\Omega}} &K^*(BT^n_+)\llbracket \mathbf{t} \rrbracket \ar[d]_-{\varphi_K}\\
\mathcal{F}_*^{T^n} \ar[r]^-{} &S^{-1}K^*(BT^n_+)\llbracket \mathbf{t} \rrbracket
} \]
has all maps injective and is a pullback square.
\end{theorem}

The coefficients of $(B\circ \Phi)[M]$ are known as \emph{equivariant $K$-theory characteristic numbers} for $M$.

By defining $\Psi$ to be the composition of $B\circ\Phi$ with the inclusion $\iota\colon\mathcal{Z}_*^{U:T^n}\to \Omega_*^{U:T^n}$, Proposition~\ref{P:Z} tells us that
\begin{equation}\label{sq}
\begin{aligned}
\xymatrix{
\mathcal{Z}_*^{U:T^n} \ar[r]^-{\Psi} \ar[d]_-{\varphi_{Z}} &K^*(BT^n_+)\llbracket \mathbf{t} \rrbracket \ar[d]_-{\varphi_K}\\
\Z[J_n] \ar[r]^-{S^{-1}\Psi} &S^{-1}K^*(BT^n_+)\llbracket \mathbf{t} \rrbracket
}
\end{aligned}
\end{equation}
is a pullback square with all maps injective and, by using~\eqref{e:utg} and~\eqref{e:Boardman}, we obtain a formula
\[
\Psi[M^{2m}]=\sum_{p\in M^{T^n}}\sigma(p)\prod_{i=1}^m \frac{\gamma_t(V_i(p)-1)}{1-V_i(p)},
\]
where we write $T_pM=V_1(p)\oplus\dots\oplus V_m(p)$, for $p\in M^{T^n}$.

\subsection{The Torus Manifold Case}

We now restrict our attention to the case where $m=n$. Recalling Theorem~\ref{T:tgraph} notice that we have a map
\begin{equation}\label{e:him}
\tilde{g}\colon \mathcal{Z}_{2n}^{U:T^n}\longrightarrow K_n,
\end{equation}
where $\tilde{g}[M]\letbe g(\Gamma_M,\alpha_M)$, that is the torus polynomial of the oriented torus graph associated to the stably complex torus manifold $M$. This can be seen to be well-defined by considering the commutative diagram
\begin{equation}\label{e:gg}
\begin{aligned}
\xymatrix{
& \mathcal{Z}_{2n}^{U:T^n} \ar[ld]_-{\tilde{g}} \ar[rd]^-{\varphi_Z} & \\
K_n \ar[rr]^-{f} & & \Z[J_n]
}
\end{aligned}
\end{equation}
where $f(s_1\wedge\dots\wedge s_n)\letbe \det [s_1 \cdots s_n] s_1\cdots s_n$ for a faithful monomial $s_1\wedge\dots\wedge s_n\in \Lambda^n(J_n)$. It also follows that $\tilde{g}$ is monic as $\varphi_Z$ is. We now consider the following diagram in our bid to show that $\tilde{g}$ is an isomorphism:

\begin{equation}
\begin{aligned}
\xymatrix{
&\mathcal{Z}_{2n}^{U:T^n} \ar[rr]^-{\Psi} \ar[dd]_-{\varphi_{Z}} \ar[ld]_-{\tilde{g}} & &K^*(BT^n_+)\llbracket \mathbf{t} \rrbracket \ar[dd]_-{\varphi_K}\\
K_n \ar[rd]_-{f} & & &\\
&\Z[J_n] \ar[rr]^-{S^{-1}\Psi} & &S^{-1}K^*(BT^n_+)\llbracket \mathbf{t} \rrbracket
}
\end{aligned}
\end{equation}

Let $h\in K_n$. Then by Theorem~\ref{T:tgraph} there is an oriented $n$-torus graph $(\Gamma,\alpha,\sigma)$ such that $h=g(\Gamma,\alpha)$. We can write
\[
h=\sum_{p\in \mathcal{V}(\Gamma)} \alpha(e_1^p)\wedge\dots\wedge \alpha(e_n^p),
\]
where we order $\mathcal{E}(\Gamma)_p=\{ e_1^p,\dots,e_n^p\}$ so that $\det [\alpha(e_1^p)\cdots\alpha(e_n^p)]=\sigma(p)$. We then have that
\begin{equation}\label{f}
f(h)=\sum_{p\in \mathcal{V}(\Gamma)} \sigma(p) \prod_{i=1}^n \alpha(e_i^p)\in \Z[J_n]
\end{equation}
and
\begin{equation}\label{e:Psi}
(S^{-1}\Psi\circ f)(h)=\sum_{p\in \mathcal{V}(\Gamma)} \sigma(p) \prod_{i=1}^n \frac{\gamma_t(\alpha(e_i^p)-1)}{1-\alpha(e_i^p)}\in S^{-1}K^*(BT^n_+)\llbracket \mathbf{t} \rrbracket.
\end{equation}

The following theorem tells us that any exterior polynomial $f(h)$ coming from an oriented torus graph, as in~\eqref{f}, is in the image of $\varphi_Z$. The proof is an adaption of~\cite[Theorem 1.1]{GZ01} from GKM-graphs to torus graphs.

\begin{theorem}\label{T:psi}
Every polynomial $h\in K_n$ satisfies
\[
(S^{-1}\Psi\circ f)(h)\in K^*(BT^n_+)\llbracket \mathbf{t} \rrbracket.
\]
\end{theorem}

\begin{proof}
Let $\alpha_1,\dots,\alpha_N$ be elements of $\Hom (T^n,S^1)$ such that for every $e\in \mathcal{E}(\Gamma)$, there exists a unique $k\in \{ 1,\dots,N\}$, such that $\alpha(e)$ is a multiple of $\alpha_k$. If $m_1\alpha_1,\dots,m_s\alpha_1$ are all the occurrences of multiples of $\alpha_1$ among all weights, let $M_1\letbe \text{lcm}\, (m_1,\dots,m_s)$. Similarly define $M_2,\dots,M_n$. Then
\[
(S^{-1}\Psi\circ f)(h)=\frac{l}{\prod_{j=1}^N (1-M_j\alpha_j)}\ ,
\]
with $l\in K^*(BT^n_+)\llbracket \mathbf{t} \rrbracket$. We will show that $(1-M_1\alpha_1)$ divides $l$ in $K^*(BT^n_+)\llbracket \mathbf{t} \rrbracket$.

We divide the vertices of $\Gamma$ into two.
\begin{enumerate}
\item Let $V_1$ be the subset that contains the vertices $p\in \mathcal{V}(\Gamma)$ for which $\alpha(e_i^p)$ is not a multiple of $\alpha_1$, for $1\leq i\leq n$.
\item Let $V_2\letbe \mathcal{V}(\Gamma)\smallsetminus V_1$.
\end{enumerate}

The part of~\eqref{e:Psi} that corresponds to vertices in $V_1$ is then of the form
\begin{equation}\label{v1}
\sum_{p\in V_1}\sigma(p)\prod_{i=1}^n  \frac{\gamma_t(\alpha(e_i^p)-1)}{1-\alpha(e_i^p)}=\frac{l_1}{\prod_{j=2}^N (1-M_j\alpha_j)}
\end{equation}
with $l_1\in K^*(BT^n_+)\llbracket \mathbf{t} \rrbracket$.

If $p\in V_2$, then there exists an edge $e$ with $i(e)=p$ such that $\alpha(e)=m\alpha_1$, with $m\in \Z\smallsetminus \{ 0\}$. Note that there is exactly one edge in $\mathcal{E}(\Gamma)_p$ with this property. Let $q=t(e)$. Since $\alpha(\bar{e})=\pm \alpha(e)\in \Hom(T^n,S^1)$ we have that $q\in V_2$. So the vertices in $V_2$ can be grouped in pairs $(p=i(e),q=t(e))$ with $\alpha(e)$ a multiple of $\alpha_1$.

Write
\[
\mathcal{E}(\Gamma)_p=\{ e_1^p,\dots,e_n^p\}\quad\text{and}\quad \mathcal{E}(\Gamma)_q=\{ e_1^q,\dots,e_n^q\}
\]
such that $e=e_1^p$, $\bar{e}=e_1^q$ and $\alpha(e_i^q)\equiv \alpha(e_i^p)\ \text{mod}\ \alpha(e)$, for $2\leq i\leq n$. We then have that
\begin{equation}\label{cong1}
1-\alpha(e_i^p)\equiv 1-\alpha(e_i^q)\ (\text{mod}\ 1-\alpha(e))\in K^*(BT^n_+),
\end{equation}
and
\[
\gamma_t(\alpha(e_i^q)-1)\equiv \gamma_t(\alpha(e_i^p)-1)\ (\text{mod}\ 1-\alpha(e))\in K^*(BT^n_+)\llbracket \mathbf{t}\rrbracket.
\]
Setting
\[
f_p=\prod_{i=1}^n \gamma_t(\alpha(e_i^p)-1)\quad\text{and}\quad f_q=\prod_{i=1}^n \gamma_t(\alpha(e_i^q)-1).
\]
we deduce that
\begin{equation}\label{cong2}
f_q\equiv f_p\ (\text{mod}\ 1-\alpha(e)).
\end{equation}
The part of~\eqref{e:Psi} corresponding to $p$ and $q$ is
\begin{equation}\label{pq}
\sigma(p)\frac{f_p}{\prod_{i=1}^n (1-\alpha(e_i^p))}+\sigma(q)\frac{f_q}{\prod_{i=1}^n (1-\alpha(e_i^q))}.
\end{equation}

\paragraph{\emph{Case 1}}

Suppose $\sigma(p)=\sigma(q)$. This implies that $\alpha(e)=-\alpha(\bar{e})\in \Hom(T^n,S^1)$ and~\eqref{pq} can be expressed as
\begin{equation}\label{pq2}
\sigma(p)\frac{f_p\prod_{i=2}^n (1-\alpha(e_i^q))-\alpha(e)f_q\prod_{i=2}^n (1-\alpha(e_i^p))}{(1-\alpha(e))\prod_{i=2}^n (1-\alpha(e_i^p))(1-\alpha(e_i^q))}.
\end{equation}
Then the congruences~\eqref{cong1} and~\eqref{cong2} imply that $(1-\alpha(e))$ divides the numerator of~\eqref{pq2}.

\paragraph{\emph{Case 2}}

Suppose $\sigma(p)\ne\sigma(q)$. This implies that $\alpha(e)=\alpha(\bar{e})\in \Hom(T^n,S^1)$ and~\eqref{pq} can be expressed as
\begin{equation}\label{pq3}
\sigma(p)\frac{f_p\prod_{i=2}^n (1-\alpha(e_i^q))-f_q\prod_{i=2}^n (1-\alpha(e_i^p))}{(1-\alpha(e))\prod_{i=2}^n (1-\alpha(e_i^p))(1-\alpha(e_i^q))}.
\end{equation}
Then the congruences~\eqref{cong1} and~\eqref{cong2} imply that $(1-\alpha(e))$ divides the numerator of~\eqref{pq3}.

We can therefore deduce that~\eqref{pq} can be written as
\[
\frac{l_{pq}}{\prod_{j=2}^N (1-M_j\alpha_j)}
\]
with $l_{pq}\in K^*(BT^n_+)\llbracket \mathbf{t}\rrbracket$. Therefore,
\begin{equation}\label{v2}
\sum_{p\in V_2} \sigma(p)\prod_{i=1}^n  \frac{\gamma_t(\alpha(e_i^p)-1)}{1-\alpha(e_i^p)}=\frac{l_2}{\prod_{j=2}^N (1-M_j\alpha_j)}
\end{equation}
with $l_2\in K^*(BT^n_+)\llbracket \mathbf{t}\rrbracket$. Adding~\eqref{v1} and~\eqref{v2} we obtain
\[
(S^{-1}\Psi\circ f)(h)=\frac{l}{\prod_{j=1}^N (1-M_j\alpha_j)}=\frac{l_1+l_2}{\prod_{j=2}^N (1-M_j\alpha_j)}
\]
with $l_1+l_2\in K^*(BT^n_+)\llbracket \mathbf{t}\rrbracket$; hence $(1-M_1\alpha_1)$ divides $l$. The same argument can be used to show that each $(1-M_j\alpha_j)$ divides $h$. The proof now follows from the following lemma:

\begin{lemma}[{\cite[Lemma 4.2]{GZ01}}]
If $P\in K^*(BT^n_+)$ and $\beta_1,\dots,\beta_n$ are linearly independent weights such that $(1-\beta_j)$ divides $P$ for all $j=1,\dots,n$, then
\[
(1-\beta_1)\cdots(1-\beta_n)\ \text{divides}\ P.
\]
\end{lemma}
\end{proof}

We can now deduce:

\begin{corollary}\label{C:*}
The homomorphism of abelian groups
\[
\tilde{g}\colon \mathcal{Z}_{2n}^{U:T^n}\longrightarrow K_n
\]
is an isomorphism, for all $n\geq 0$.
\end{corollary}

\begin{proof}
We know that $\tilde{g}$ is injective as $\varphi_Z$ is, see~\eqref{e:gg}. By Theorem~\ref{T:psi}
\[
(S^{-1}\Psi\circ f)(h)\in K^*(BT^n_+)\llbracket \mathbf{t} \rrbracket,
\]
for every $h\in K_n$, and since~\eqref{sq} is a pullback square we have that $h\in \tilde{g}(\mathcal{Z}_{2n}^{U:T^n})$. Therefore, $\tilde{g}$ is surjective.
\end{proof}

Define the graded ring
\begin{equation}\label{eq:Xi}
\Xi_*\letbe \bigoplus_{n\geq 0}\mathcal{Z}_{2n}^{U:T^n},
\end{equation}
where the multiplication is given by $[M_1^{2n_1}]\cdot [M_2^{2n_2}]=[M^{2n_1}_1\times M^{2n_2}_2]$ in a natural way such that the $T^{n_1+n_2}=T^{n_1}\times T^{n_2}$-action is given by
\[
((t_1,t_2),(x_1,x_2))\longmapsto (t_1x_1,t_2x_2),
\]
for $x_i\in M_i$ and $t_i\in T^{n_i}$. We note that this multiplication depends on the ordering of the cartesian product of $M^{n_1}$ with $T^{n_1}$-action and $M^{n_2}$ with $T^{n_2}$-action. The $T^{n_1+n_2}$-action on $M_1^{2n_1}\times M_2^{2n_2}$ and $M_2^{2n_2}\times M_1^{2n_1}$ are not equivariantly bordant, except when $[M_1^{2n_1}]=[M_2^{2n_2}]$, but differ up to an automorphism of $T^{n_1+n_2}$. We therefore obtain that $\Xi_*$ is a non-commutative graded ring with unit given by a point.

Similarly, let
\[
K_*\letbe \bigoplus_{n\geq 0} K_n
\]
with the obvious product structure and observe that this too is a non-commutative graded ring with unit. Following from Corollary~\ref{C:*} we see that:

\begin{theorem}\label{T:g}
The homomorphism of non-commutative graded rings
\[
\tilde{g}\colon \Xi_*\longrightarrow K_*
\]
is an isomorphism.
\end{theorem}

Notice that a non-zero exterior polynomial $h\in K_n$ must have at least $n+1$ monomials since $d(h^*)=0$. This immediately leads to the following result:

\begin{corollary}
As a strict lower bound, $n+1$ is the minimum number of fixed points for a non-bounding stably complex torus manifold of dimension $2n$.
\end{corollary}

\begin{proof}
A non-bounding stably complex torus manifold $M^{2n}$ gives a non-zero element in $\mathcal{Z}_{2n}^{U:T^n}$ whose image under $\tilde{g}$ is a non-zero exterior polynomial $h\in K_n$. Since the minimum number of monomials $h$ can have is $n+1$ we have that $M$ has at least $n+1$ fixed points.
\end{proof}

\section{Omnioriented Quasitoric Manifolds}\label{OQTM}

Davis and Januszkiewicz~\cite{DJ91} defined what are now known as quasitoric manifolds in an attempt to find a topological analogue to toric varieties. They did not insist these be smooth manifolds but it has been shown, see~\cite{BP02}, that they do always admit a smooth structure. We therefore use the following as our working definition:

\begin{definition}
A \emph{quasitoric manifold} is an even-dimensional smooth closed manifold $M^{2n}$ with a locally standard smooth $T^n$-action such that the orbit space is a simple polytope $P$.
\end{definition}

Each quasitoric manifold determines a \emph{characteristic map} $\lambda$ on $P$, which sends the facets $F_j$ of $P$ onto non-trivial elements of $\Hom(S^1,T^n)$, unique up to sign, such that the $n$ facets of $P$ meeting at a vertex are mapped to a basis of $\Hom(S^1,T^n)\cong \Z^n$. We think of the pair $(P,\lambda)$ as the \emph{combinatorial data} associated to the quasitoric manifold $M$. Davis and Januszkiewicz~\cite{DJ91} showed that given a pair $(P,\lambda)$ of a simple polytope $P$ and a characteristic map $\lambda$, satisfying the above condition, we can construct a quasitoric manifold $\mathcal{M}(P,\lambda)$ with $(P,\lambda)$ as its combinatorial data.

We now give a brief exposition on the relationship between quasitoric manifolds and their combinatorial data. A more in-depth discussion can be found in~\cite{BPR07}.

\subsection{Quasitoric Pairs}

An $n$-dimensional convex polytope $P$ is the bounded intersection of $m$ irredundant half-spaces in $\R^n$. The bounding hyperplanes $H_1,\dots,H_m$ intersect $P$ in its facets (codimension-one faces) $F_1,\dots,F_m$ and $P$ is called \emph{simple} if each vertex is the intersection of exactly $n$ facets. Two polytopes are \emph{combinatorially equivalent} if their face posets are isomorphic and we call the corresponding equivalence classes \emph{combinatorial polytopes}.

Associated to each combinatorial simple $n$-polytope $P$, with $m$ facets, is an $(n+m)$-dimensional smooth equivariantly framed $T^m$-manifold $Z_P$ called the \emph{moment-angle manifold} associated to $P$. It can be defined as the pullback
\begin{equation}\label{E:pbk}
\begin{aligned}
\xymatrix{
Z_P \ar[r]^-{i_Z} \ar[d] &\C^m \ar[d]^-{\rho}\\
P \ar[r]^-{i_P} &\R^m_{\geq}
}
\end{aligned}
\end{equation}
where $i_P$ is the canonical affine embedding which maps a point of $P$ to its $m$-vector of distances from the hyperplanes $H_1,\dots,H_m$ and
\[
\rho(z_1,\dots ,z_m)\letbe (|z_1|^2,\dots ,|z_m|^2).
\]
The vertical maps are projections onto the $T^m$-orbit spaces and $i_Z$ is an equivariant embedding. Since $Z_P$ can be expressed as the complete intersection of $m-n$ real quadratic hypersurfaces in $\C^m$, see~\cite{BPR07}, it is smooth and equivariantly framed so we have a $T^m$-equivariant isomorphism
\begin{equation}\label{e:tauZ}
\tau(Z_P)\oplus \nu(i_Z)\cong Z_P\times \C^m.
\end{equation}

\begin{definition}\label{D:qpair}
A \emph{quasitoric pair} $(P,\Lambda)$ consists of a combinatorial oriented simple $n$-polytope $P$ and an integral $(n\times m)$-matrix $\Lambda$ whose columns $\lambda_i$ satisfy
\begin{align*}\tag{$\star$}
\det (\lambda_{i_1}\cdots \lambda_{i_n})=\pm 1,\quad \text{whenever $F_{i_1}\cap\dots\cap F_{i_n}$ is a vertex of $P$.}
\end{align*}
We define an equivalence class on quasitoric pairs whereby $(P_1,\Lambda_1)$ is equivalent to $(P_2,\Lambda_2)$ if and only if $P_1=P_2$ and there exists an $(m\times m)$-permutation matrix $\Sigma$ such that $\Lambda_1=\Lambda_2 \Sigma$; the matrix $\Sigma$ may be thought of as allowing for a reordering of the facets of $P$.
\end{definition}

From a quasitoric pair $(P^n,\Lambda)$ we can construct a quasitoric manifold $M^{2n}=\mathcal{M}(P,\Lambda)$ in the following manner: consider $\Lambda$ as an epimorphism $\Lambda\colon T^m\to T^n$, whose kernel we write as $K(\Lambda)$, which is isomorphic to $T^{m-n}$ by Condition~($\star$), and acts smoothly and freely on $Z_P$. Therefore, the orbit space
\[
M^{2n}=\mathcal{M}(P,\Lambda)\letbe Z_P/K(\Lambda)
\]
is a $2n$-dimensional smooth manifold with a smooth action of the quotient $n$-torus $T^{n}\cong T^m/K(\Lambda)$ whose orbit space is $P$. Davis and Januszkiewicz~\cite{DJ91} show that the action is locally standard. We also have a canonical orientation of $M$ induced by the orientation of $P$.

By factoring out the decomposition~\eqref{e:tauZ} the quasitoric manifold $M=\mathcal{M}(P,\Lambda)$ also has a canonical tangentially stably complex $T^n$-structure
\[
\tau(M)\oplus \R^{2(m-n)}\cong \rho_1\oplus\dots\oplus \rho_m,
\]
where $\rho_i$ is the complex line bundle defined by the projection
\[
Z_P\times_{K(\Lambda)} \C_i\longrightarrow M
\]
associated to the action of $K(\Lambda)$ on the $i^{\text{th}}$ coordinate subspace.

The characteristic submanifolds $M_i=\pi^{-1}(F_i)\subset M$, $1\leq i\leq m$, have orientable normal bundles $\nu_i$ that are isomorphic to $\rho_i|_{M_i}$ and $\rho_i$ is trivial over $M\smallsetminus M_i$. Therefore, the orientations underlying the complex structures on the bundles $\rho_i$ induce orientations on the characteristic submanifolds $M_i$.

\begin{definition}
An \emph{omniorientation} of a quasitoric manifold $M$ is a choice of orientations for $M$ and the characteristic submanifolds $M_1,\dots,M_m$.
\end{definition}

Thus, from a quasitoric pair $(P,\Lambda)$ we can construct an omnioriented quasitoric manifold that has a canonical tangentially stably complex $T^n$-structure.

Now consider an omnioriented quasitoric manifold $M^{2n}$. By definition there is a smooth projection $\pi\colon M\to P$ onto a simple $n$-polytope $P=P(M)$ which is oriented by the underlying orientation of $M$. The orientation of the facial submanifolds $M_i=\pi^{-1}(F_i)$, for the facets $F_1,\dots,F_m$ of $P$, have their normal bundles oriented by the omniorientation of $M$. The isotropy subcircles $T_i\leq T^n$, $1\leq i\leq m$, are oriented and therefore specify $m$ vectors in $\Z^n$, which form the columns of an $(n\times m)$-matrix $\Lambda(M)$ that satisfies Condition~$(\star)$. So the pair $(P(M),\Lambda(M))$ is a quasitoric pair. The omnioriented quasitoric manifold $\mathcal{M}(P(M),\Lambda(M))$ is equivariantly diffeomorphic to $M$ and the quasitoric pair associated to $\mathcal{M}(P,\Lambda)$ is $(P,\Lambda)$ itself unless the facets have been labelled in a different order in which case we can apply an $(m\times m)$-permutation matrix $\Sigma$ to $\Lambda$ and recall that $(P,\Lambda)$ is equivalent to $(P,\Lambda\Sigma)$.

We therefore obtain

\begin{corollary}
There is a bijection between the set of quasitoric pairs and the set of omnioriented quasitoric manifolds.
\end{corollary}

\subsection{Ring of Quasitoric Pairs}

Suppose we are given two quasitoric pairs $(P_1^{n_1},\lambda_1)$ and $(P_2^{n_2},\lambda_2)$, where
\[
\mathcal{F}(P_1)=\{ F_1,\dots,F_{m_1}\}\quad \text{and}\quad \mathcal{F}(P_2)=\{ F_1',\dots,F_{m_2}'\}.
\]
We define a product
\[
(P_1,\lambda_1)\times (P_2,\lambda_2)\letbe (P_1\times P_2,\lambda_1\times \lambda_2),
\]
where the characteristic map is defined as
\[
(\lambda_1\times \lambda_2)(F_i\times P_2)=(\lambda_1(F_i),0,\dots,0)\quad \text{and}\quad (\lambda_1\times \lambda_2)(P_1\times F_i')=(0,\dots,0,\lambda_2(F_i')).
\]

\begin{definition}
We denote the free abelian group generated by all quasitoric pairs by $\mathcal{Q}_*$, where we may interpret $+$ as disjoint union and grade $\mathcal{Q}_*$ by the dimension of the polytope. The multiplication depends on the ordering of $P_1\times P_2$ so $\mathcal{Q}_*$ forms a graded non-commutative ring.
\end{definition}

We have a homomorphism of non-commutative graded rings
\begin{equation}\label{eq:M}
\mathcal{M}\colon \mathcal{Q}_*\longrightarrow \Xi_*,
\end{equation}
by constructing the omnioriented quasitoric manifold associated to a quasitoric pair. This is clearly not injective as can be shown by considering any bounding pair. For example, consider the quasitoric pair $(\Delta^1,\lambda)$, where $\mathcal{F}(\Delta^1)=\{ D_1,D_2\}$, such that $\lambda(D_1)=\lambda(D_2)$. The sign of the two vertices of $\Delta^1$, which are also the facets, are then distinct and the torus polynomial of the oriented torus graph related to $(\Delta^1,\lambda)$ is zero. Therefore, $\mathcal{M}(\Delta^1,\lambda)=0\in \Xi_1$.

Suppose $(P,\lambda)\in \mathcal{Q}_*$. For each vertex $v\in P$ define the ordered set
\begin{equation}\label{E:F(P)}
\mathcal{F}(P)_v=\{ F_{v_1},\dots,F_{v_n}\}
\end{equation}
to be the facets in $\mathcal{F}(P)$ such that $v=F_{v_1}\cap\dots\cap F_{v_n}$ and that the inward pointing normals of the $F_{v_i}$ form a positive basis in the standard orientation of $\R^n$. In~\cite[Remark 5.17]{BPR10} it is shown that
\begin{equation}\label{E:Psigma}
\det [\lambda(F_{v_1})\cdots \lambda(F_{v_n})]=\sigma(v),
\end{equation}
where $\sigma(v)$ is the sign of $v$ as defined in Definition~\ref{D:sign}.

\begin{definition}
The \emph{quasitoric polynomial} of the quasitoric pair $(P,\lambda)$ is the faithful exterior polynomial
\[
\mathfrak{g}{(P,\lambda)}=\sum_{v\in P}\lambda(\mathcal{F}(P)_v)\in \Lambda^n(J_n^*),
\]
where $\lambda(\mathcal{F}(P)_v)\letbe \lambda(F_{v_1})\wedge\dots\wedge \lambda(F_{v_n})$.
\end{definition}

The torus graph $(\Gamma_M,\alpha_M)$ of a quasitoric manifold $M=\mathcal{M}(P,\lambda)$ is given by the 1-skeleton of $P$ with axial function $\alpha_M$ given by
\[
[\alpha_M(e_{v_1})\cdots \alpha_M(e_{v_n})]=[\lambda(F_{v_1})\cdots \lambda(F_{v_n})]^*,
\]
using the notation of~\eqref{matrixdual}, where $e_{v_i}$ is the unique edge of $P$ such that $F_{v_i}\cap e_{v_i}=v$.
We can therefore deduce that the torus polynomial $g{(\Gamma_M,\alpha_M)}$ of the oriented torus graph coming from $M=\mathcal{M}(P,\lambda)$ is dual to the quasitoric polynomial of $(P,\lambda)$, that is
\begin{equation}\label{E:dual}
g{(\Gamma_M,\alpha_M)}=(\mathfrak{g}{(P,\lambda)})^*.
\end{equation}

\begin{corollary}
For every quasitoric pair $(P,\lambda)$, we have
\[
d(\mathfrak{g}{(P,\lambda)})=0.
\]
\end{corollary}

\begin{proof}
This follows directly from Theorem~\ref{T:tgraph} and~\eqref{E:dual}.
\end{proof}

\begin{definition}
Let $\mathfrak{K}_n$ denote the abelian group of all faithful exterior polynomials $h\in \Lambda^n(J_n^*)$ such that $d(h)=0$.
\end{definition}

\begin{remark}\label{R:mfg}
As for~\eqref{e:him}, we may regard $\mathfrak{g}$ as a homomorphism
\begin{equation}\label{eq:mfg}
\mathfrak{g}\colon \mathcal{Q}_n\longrightarrow \mathfrak{K}_n
\end{equation}
by taking the quasitoric polynomial of a quasitoric pair.
\end{remark}

\begin{definition}
Let $(P,\lambda)\in\mathcal{Q}_{k}$, for some $1\leq k<n$. An \emph{$n$-quasitoric polynomial $\iota(\mathfrak{g}{(P,\lambda)})$} is the image of $\mathfrak{g}{(P,\lambda)}$ under a monomorphism
\[
\iota\colon\text{Hom}(S^1,T^k)\longrightarrow \text{Hom}(S^1,T^n).
\]
\end{definition}

Given two quasitoric pairs $(P^{n_1},\lambda_1)$ and $(P^{n_2},\lambda_2)$ using the two monomorphisms
\begin{align*}
\Z^{n_1}\cong\Hom(S^1,T^{n_1})&\xrightarrow[\phantom{\qquad}]{\iota_1} \Hom(S^1,T^{n_1+n_2})\cong \Z^{n_1+n_2}\\
\Z^{n_2}\cong\Hom(S^1,T^{n_2})&\xrightarrow[\phantom{\qquad}]{\iota_2} \Hom(S^1,T^{n_1+n_2})\cong \Z^{n_1+n_2}
\end{align*}
thought of as inclusion of the first $n_1$ factors and last $n_2$ factors respectively, we obtain two $(n_1+n_2)$-quasitoric polynomials $\iota_1(\mathfrak{g}{(P_1^{n_1},\lambda_1)})$ and $\iota_2(\mathfrak{g}{(P_2^{n_2},\lambda_2)})$. It is easy to see that:

\begin{lemma}
We have a product formula:
\[
\mathfrak{g}{(P_1^{n_1}\times P_2^{n_2},\lambda_1\times\lambda_2)}=\iota_1(\mathfrak{g}{(P_1^{n_1},\lambda_1)})\iota_2(\mathfrak{g}{(P_2^{n_2},\lambda_2)}).
\]
\end{lemma}

We can now use several results to construct a commutative diagram of non-commutative graded rings. By taking
\[
K_*\letbe \bigoplus_{n\geq 0} K_n\quad \cong \quad \mathfrak{K}_*\letbe \bigoplus_{n\geq 0} \mathfrak{K}_n
\]
we obtain isomorphic non-commutative graded rings. We obtain the following commutative diagram of non-commutative graded rings
\begin{equation}\label{E:commdiag}
\begin{aligned}
\xymatrix{
\mathcal{Q}_* \ar[r]^-{\mathcal{M}} \ar[d]_-{\mathfrak{g}} &\Xi_* \ar[d]^-{\tilde{g}}\\
\mathfrak{K}_* \ar@{<->}[r]^-{\cong} &K_*
}
\end{aligned}
\end{equation}
using~\eqref{eq:M},~\eqref{e:him} and Remark~\ref{R:mfg}. Since, by Theorem~\ref{T:g}, $\tilde{g}$ is an isomorphism we have that $\Xi_*\cong \mathfrak{K}_*$ and the following commutative diagram of non-commutative rings:
\[ \xymatrix{
&\Xi_* \ar[dd]^-{\cong}\\
\mathcal{Q}_* \ar[ur]^-{\mathcal{M}} \ar[dr]_-{\mathfrak{g}}&\\
&\mathfrak{K}_*
} \]

The study the homomorphism $\mathcal{M}$ is therefore equivalent to the study of $\mathfrak{g}$. We know that $\mathcal{M}$ and $\mathfrak{g}$ are not monomorphisms and are interested whether or not these homomorphisms are epimorphisms and will use this section to investigate this matter. 

Given a non-zero polynomial $h\in \Lambda^k(J_n^*)$, for $k>0$, we write $h$ uniquely as a finite sum of monomials
\begin{equation}\label{h}
h=h_1+\dots+h_r.
\end{equation}
If $d(h)=0$, then it may be possible to split the monomials up so we can form two non-zero polynomials $f_1=\sum_{i\in A}h_i$ and $f_2=\sum_{i\in B}h_i$ such that $A\cup B=\{1,\dots,r\}$, $A\cap B=\emptyset$ and $d(f_1)=0=d(f_2)$.

\begin{definition}
Given a non-zero polynomial $h\in \Lambda^k(J_n^*)$, for $k>0$, such that \text{$d(h)=0$}, then $h$ is \emph{disconnected} if we can split its monomials as $h=f_1+f_2$  such that
\[
d(f_1)=0=d(f_2),
\]
where $f_1,f_2\ne 0$. If this is not possible, then $h$ is \emph{connected}.
\end{definition}

\begin{example}
For indeterminates $s_i\in J_n^*$, we give the following examples:
\begin{enumerate}
\item $h=s_1\wedge s_2+s_2\wedge s_3+s_3\wedge s_1$ is connected.
\item $h=s_1-s_2+s_3-s_4$ is disconnected.
\end{enumerate}
\end{example}

We will now consider the surjectivity of $\mathcal{M}$ for low dimensional cases.

\subsubsection*{The 1-dimensional Case}

Any connected polynomial $h\in \mathfrak{K}_1$ consists of the two elements $t,\bar{t}\in J_1^*$ that form a basis of $\Hom(S^1,T^n)$  and is of the form $\pm(t-\bar{t})$. This can easily be seen to be the quasitoric polynomial of $(\Delta^1,\lambda_+)$ or $(\Delta^1,\lambda_-)$, where
\[
\lambda_+(D_1)=\lambda_-(D_2)=t\quad \text{and}\quad \lambda_+(D_2)=\lambda_-(D_1)=\bar{t}.
\]
Therefore every $h\in\mathfrak{K}_1$ is of the form $n(t-\bar{t})$, for some $n\in\Z$, and is the image of $n$ disjoint copies of $(\Delta^1,\lambda_+)$, if $n$ is positive, or $n$ disjoint copies of $(\Delta^1,\lambda_-)$, if $n$ is negative. We therefore have:

\begin{corollary}
The homomorphism of abelian groups
\[
\mathcal{M}\colon \mathcal{Q}_1\longrightarrow \Xi_1
\]
is surjective.
\end{corollary}

In~\cite[\S 6]{BR01} it was shown how to take the connected sum of two simple $n$-polytopes $P$ and $Q$ at distinguished vertices, $v$ and $w$ respectively, to obtain another simple $n$\nobreakdash-\hspace{0pt}polytope $P\#_{v,w} Q$. Informally, we ``cut off'' $v$ from $P$ and $w$ from $Q$ and then, after a projective transformation, ``glue'' the rest of $P$ to the rest of $Q$ along the two new simplex facets to obtain $P\#_{v,w} Q$. If we have two quasitoric pairs $(P,\lambda)$ and $(Q,\eta)$, with vertices $v\in P$ and $w\in Q$ such that
\[
\lambda(\mathcal{F}(P)_{v})+\eta(\mathcal{F}(Q)_{w})=0,
\]
which implies that $\sigma(v)\ne \sigma(w)$, then we can then form the quasitoric pair
\[
(P\#_{v,w}Q,\lambda\#\eta).
\]
This has the property that
\begin{equation}\label{E:csum}
\mathcal{M}(P\#_{v,w} Q,\lambda\#\eta)=\mathcal{M}(P,\lambda)+\mathcal{M}(Q,\eta),
\end{equation}
in $\Xi_n$, which is an equation of bordism classes. In this section we will exhibit this operation.

\begin{definition}
The \emph{vertex figure} of an $n$-polytope $P$ at a vertex $v$ is the \text{$(n-1)$}-polytope
\[
P/v\letbe P\cap H_v,
\]
where $H_v$ is a hyperplane that isolates the single vertex $v\in P$ and no other.
\end{definition}

As shown in~\cite[Proposition 2.4]{Z95} the combinatorial type of $P/v$ is independent of the choice of $H_v$.

\begin{proposition}[{\cite[Proposition 2.16]{Z95}}]
A polytope $P$ is a simple polytope if and only if $P/v$ is a simplex, for all vertices $v\in P$.
\end{proposition}

For a polytope $P$ we will write $m(P)$ to be the number of facets of $P$, and $v(P)$ to be the number of vertices of $P$. We assume the dimensions of our polytopes to be $\geq 2$ for now, and we will discuss the 1-dimensional case afterwards. We embed the standard $(n-1)$-simplex $\Delta^{n-1}$ in the subspace $\{ x\in \R^n\mid x_1=0\}$ and construct the \emph{polyhedral template} $\Gamma^n$ by taking the cartesian product with the first coordinate axis. So $\Gamma^n$ is the intersection of $n$ half spaces in $\R^n$ and we denote its facets by $G_i$ which have the form $\R\times D_i$, for $1\leq i\leq n$, where
\[
\mathcal{F}(\Delta^{n-1})=\{ D_1,\dots,D_n\}.
\]
Both $\Gamma^n$ and $G_i$ are divided into positive and negative halves determined by the sign of the coordinate $x_1$.

Suppose we have simple polytopes $P^n$ and $Q^n$ with distinguished vertices $v\in P$ and $w\in Q$. We write
\[
v=E_1\cap\dots\cap E_n\quad \text{and}\quad w=F_1\cap\dots\cap F_n,
\]
and
\[
\mathcal{C}_{v}=\mathcal{F}(P)\smallsetminus \{ E_1,\dots,E_n\}\quad \text{and}\quad \mathcal{C}_{w}=\mathcal{F}(Q)\smallsetminus \{ F_1,\dots,F_n\}.
\]
We now suppose that $P$ is embedded in $\R^n$ such that $v$ is at $(1,0,\dots,0)$ and that its vertex figure is mapped to $\Delta^{n-1}=\Gamma^n\cap \{ x_1=0\}$. Therefore, the hyperplanes defining the facets in $\mathcal{C}_{v}$ lie in the negative half of $\R^n$, that is the set $\{ x\mid x_1\leq 0\}$. We now apply the projective transformation
\[
\pi_1(x)=\frac{x}{1-x_1}
\]
to $P$. This sends $v$ to $+\infty$ and the hyperplane defining $E_i$ to the hyperplane defining $G_i$ (without loss of generality), for $1\leq i\leq n$. We similarly embed $Q$ such that $w$ is at $(-1,0,\dots,0)$ and the hyperplanes defining $\mathcal{C}_{w}$ lie in the positive half of $\R^n$. We then apply the projective transformation
\[
\pi_2(x)=\frac{x}{1+x_1}
\]
to $Q$, which sends $w$ to $-\infty$ and the hyperplane defining $F_i$ to the hyperplane defining $G_i$, for $1\leq i\leq n$.

\begin{definition}
The \emph{connected sum} $P\#_{v,w} Q$ of $P$ at $v$ and $Q$ at $w$ is the simple polytope determined by all the hyperplanes that define $\mathcal{C}_v$, $\mathcal{C}_w$ and $G_i$, for $1\leq i\leq n$.
\end{definition}

This is only defined up to combinatorial equivalence and obviously depends on the vertices that we choose and possibly the orderings for $E_i$ and $F_i$.

We now identify the facets $\pi_1(E_i)$ with $\pi_2(F_i)$ and call it $G_i$, for $1\leq i\leq n$. We can write
\[
\mathcal{F}(P\#_{v,w} Q)=\mathcal{C}_{v}\cup \{ G_i\mid 1\leq i\leq n\}\cup \mathcal{C}_{w}.
\]
We also have
\begin{equation}\label{eq:mq}
m(P\#_{v,w} Q)=m(P)+m(Q)-n\quad \text{and}\quad v(P\#_{v,w} Q)=v(P)+v(Q)-2,
\end{equation}
as $n$ facets of $P$ are identified with $n$ facets of $Q$ and we lose the two vertices $v$ and $w$. 

\begin{example}\label{e:simp}
The connected sum $\Delta^n\#_{v,w} Q^n$, where the symmetry of the $n$-simplex $\Delta^n$ ensures that the result is independent of the choice of $v$, is equivalent to adding an extra hyperplane $H_w$ into the defining set for $Q$ that isolates the vertex $w$ but no other vertex. This can be seen by observing that the projective transformation $\pi_1$ sends $\Delta^n$ to $\Gamma^n$, truncated in its negative half by a single hyperplane. So we may think of the connected sum at a vertex of a simple polytope $Q$ with a simplex as ``chopping off a vertex'' of $Q$.
\end{example}

Suppose we have two quasitoric pairs $(P,\lambda)$ and $(Q,\eta)$ with distinguished vertices $v\in P$ and $w\in Q$ such that
\begin{equation}\label{E:o}
\lambda(\mathcal{F}(P)_{v})+\eta(\mathcal{F}(Q)_{w})=0.
\end{equation}
This implies that, as unordered sets,
\begin{equation}\label{E:unordered}
\{ \lambda(E_1),\dots,\lambda(E_n)\}=\{ \eta(F_1),\dots,\eta(F_n)\},
\end{equation}
where $v=E_1\cap\dots\cap E_n$ and $w=F_1\cap\dots\cap F_n$, and that
\begin{equation}\label{E:ne}
\sigma(v)\ne \sigma(w).
\end{equation}
Once we have embedded $P$ in $\R^n$ as described above, that is, mapping $v$ to $(1,0,\dots,0)$ and its vertex figure $P/v$ to $\Delta^{n-1}=\Gamma^n\cap \{ x_1=0\}$, then we can induce a labelling on the facets of $\Delta^{n-1}$ by labelling $D_i\in \mathcal{F}(\Delta^{n-1})$ with $\lambda(E_i)$, for $1\leq i\leq n$. We now embed $Q$ in $\R^n$ mapping $w$ to $(-1,0,\dots,0)$ and $Q/w$ to $\Delta^{n-1}=\Gamma^n\cap \{ x_1=0\}$. By~\eqref{E:o}, we may choose a permutation of the facets $F_i$ such that the induced labelling on $\Delta^{n-1}$ by $Q/w$ agrees with the labelling induced by $P/v$; that is
\[
\lambda(E_i)=\eta(F_i),\quad \text{for}\ 1\leq i\leq n.
\]
This can be seen geometrically as the vertex $w$, with the labelling on the facets $F_1,\dots,F_n$ given by $\eta$, is locally the reflection in $\{ x_1=0\}$ of the vertex $v$, with the labelling on the facets $E_1,\dots,E_n$ given by the $\lambda$. It can easily be seen that this would not be possible if~\eqref{E:unordered} held but~\eqref{E:ne} did not as we would not have this local reflection.

We can now form the \emph{connected sum of quasitoric pairs}
\[
(P,\lambda)\#_{v,w}(Q,\eta)\letbe (P\#_{v,w} Q,\lambda\#\eta)
\]
where the characteristic function $\lambda\#\eta$ is defined as
\[
\lambda\#\eta(F)=
\begin{cases}
\lambda(F), &\text{for}\ F\in \mathcal{C}_{v};\\
\lambda(F)=\eta(F), &\text{for}\ F=G_i,\quad 1\leq i\leq n;\\
\eta(F), &\text{for}\ F\in \mathcal{C}_{w}.
\end{cases}
\]

\begin{lemma}\label{L:connsum}
We have the following formula for the quasitoric polynomial of a connected sum of quasitoric pairs:
\[
\mathfrak{g}{(P\#_{v,w}Q,\lambda\#\eta)}=\mathfrak{g}{(P,\lambda)}+\mathfrak{g}{(Q,\eta)}\in \Lambda^n(J_n^*).
\]
\end{lemma}

\begin{proof}
Using~\eqref{E:o} we have
\begin{align*}
\mathfrak{g}{(P\#_{v,w}Q,\lambda\#\eta)}&=(\mathfrak{g}{(P,\lambda)}-\lambda(\mathcal{F}(P)_{v}))+(\mathfrak{g}{(Q,\eta)}-\eta(\mathcal{F}(Q)_{w}))\\
&=\mathfrak{g}{(P,\lambda)}+\mathfrak{g}{(Q,\eta)}.
\end{align*}
\end{proof}

We can now prove the following formula which tells us, like in non-equivariant bordism, that taking connected sums is equivalent to disjoint unions.

\begin{corollary}\label{C:M}
The following holds:
\[
\mathcal{M}(P\#_{v,w} Q,\lambda\#\eta)=\mathcal{M}(P,\lambda)+\mathcal{M}(Q,\eta)
\]
in $\Xi_n$.
\end{corollary}

\begin{proof}
This follows directly from~\eqref{E:dual},~\eqref{E:commdiag} and Lemma~\ref{L:connsum}.
\end{proof}

\begin{remark}[1-Dimensional Case]\label{R:1-dim}
A combinatorial simple 1-dimensional polytope is equivalent to $\Delta^1$. This is a special case as now the vertices are also the facets. We can still form the connected sum of polytopes to obtain $\Delta^1\#_{v,w} \Delta^1=\Delta^1$, where
\[
\mathcal{F}(\Delta^1\#_{v,w} \Delta^1)=\mathcal{C}_v\cup \mathcal{C}_w.
\]
If we have two quasitoric pairs $(\Delta^1,\lambda)$ and $(\Delta^1,\eta)$ such that
\[
\lambda(\mathcal{F}(\Delta^1)_v)+\eta(\mathcal{F}(\Delta^1)_w)=0,
\]
for vertices $v,w$, then we can form the connected sum of quasitoric pairs as above and we obtain
\[
(\Delta^1,\lambda)\#_{v,w} (\Delta^1,\eta)=(\Delta^1,\lambda\#_{v,w}\eta),
\]
where the characteristic function is now given by
\[
\lambda\#\eta(F)=
\begin{cases}
\lambda(F), &\text{for $F\in \mathcal{C}_v$};\\
\eta(F), &\text{for $F\in \mathcal{C}_w$}.
\end{cases}
\]
The formula of Corollary~\ref{C:M} obviously still holds in this case.
\end{remark}

We now prove a proposition which, for $n\geq 2$, allows us to take two quasitoric pairs $(P_1^n,\lambda_1)$ and $(P_2^n,\lambda_2)$ and form a quasitoric pair $(P,\lambda)$ such that
\[
\mathfrak{g}(P,\lambda)=\mathfrak{g}(P_1^n,\lambda_1)+\mathfrak{g}(P_2^n,\lambda_2).
\]

\begin{proposition}\label{P:add}
Suppose $h_1$ and $h_2$ are the quasitoric polynomials of quasitoric pairs $(P_1^n,\lambda_1)$ and $(P_2^n,\lambda_2)$ respectively, for $n>1$. Then $h_1+h_2$ is the quasitoric polynomial of a quasitoric pair.
\end{proposition}

\begin{proof}
If there are vertices $v_1\in P_1$ and $v_2\in P_2$ such that
\[
\lambda_1(\mathcal{F}(P_1)_{v_1})=s_1\wedge s_2\wedge\dots\wedge s_n\quad \text{and}\quad \lambda_2(\mathcal{F}(P_2)_{v_2})=-s_1\wedge s_2\wedge\dots\wedge s_n
\]
which implies
\begin{equation}\label{eq:pm}
\lambda_1(\mathcal{F}(P_1)_{v_1})+\lambda_2(\mathcal{F}(P_2)_{v_2})=0,
\end{equation}
then we can form the connected sum of $(P_1,\lambda_1)$ and $(P_2,\lambda_2)$ at $v_1$ and $v_2$ to obtain the quasitoric pair $(P_1\#_{v_1,v_2}P_2,\lambda)$, where
\[
\mathfrak{g}{(P_1\#_{v_1,v_2}P_2,\lambda)}=\mathfrak{g}{(P_1,\lambda_1)}+\mathfrak{g}{(P_2,\lambda_2)}.
\]
Therefore, $h_1+h_2$ is the quasitoric polynomial of $(P_1\#_{v_1,v_2}P_2,\lambda)$.

Now suppose there are no such vertices satisfying~\eqref{eq:pm}. If $h_1$ contains a monomial $s_1\wedge s_2\wedge\dots\wedge s_n$ but $h_2$ contains a monomial $-\widetilde{s_1}\wedge s_2\wedge\dots\wedge s_n$, then there must be vertices $v_1\in P_1$ and $v_2\in P_2$ such that
\[
\lambda_1(\mathcal{F}(P_1)_{v_1})=s_1\wedge s_2\wedge\dots\wedge s_n\quad \text{and}\quad \lambda_2(\mathcal{F}(P_2)_{v_2})=-\widetilde{s_1}\wedge s_2\wedge\dots\wedge s_n.
\]
Consider the quasitoric pair $(Q,\eta)$, where $Q=\Delta^1\times \Delta^{n-1}\subset \R^1\times \R^{n-1}$ with the characteristic function
\[
\eta(F)=
\begin{cases}
\widetilde{s_1}, &F=\{0\}\times \Delta^{n-1};\\
s_1, &F=\{1\}\times \Delta^{n-1};\\
s_{i+1}, &F=\Delta^1\times D_i,\ \text{for}\ 1\leq i<n;\\
\mathfrak{s}, &F=\Delta^1\times D_n,
\end{cases}
\]
where $\mathcal{F}(\Delta^{n-1})=\{ D_1,\dots,D_n\}$ and $\mathfrak{s}:=s_2+\dots+s_n\in \Z^n$, which represents an element of $J_n^*$. Then define the quasitoric pair $(Q,\widetilde{\eta})$ where the characteristic function is given by
\[
\widetilde{\eta}(F)=
\begin{cases}
s_1, &F=\{0\}\times \Delta^{n-1};\\
\widetilde{s_1}, &F=\{1\}\times \Delta^{n-1};\\
s_{i+1}, &F=\Delta^1\times D_i,\ \text{for}\ 1\leq i<n;\\
\mathfrak{s}, &F=\Delta^1\times D_n.
\end{cases}
\]
Notice that
\begin{equation}\label{E:Q}
\mathfrak{g}(Q,\eta)+\mathfrak{g}(Q,\widetilde{\eta})=0
\end{equation}
as the pairs $(Q,\eta)$ and $(Q,\widetilde{\eta})$ have the same base polytope and the same indeterminates labelling its facets but with opposite orientations. So we can form the connected sum
\begin{equation}\label{E:Q1}
(W,\beta)\letbe (Q,\eta)\#_{q,\tilde{q}} (Q,\widetilde{\eta})
\end{equation}
at vertices $q$ and $\tilde{q}$ such that
\[
\eta(\mathcal{F}(Q)_q)+\widetilde{\eta}(\mathcal{F}(Q)_{\tilde{q}})=0.
\]
We assume that the monomials $\eta(\mathcal{F}(Q)_q)$ and $\widetilde{\eta}(\mathcal{F}(Q)_{\tilde{q}})$ are made up of the indeterminates $s_1,s_3,\dots,s_n,\mathfrak{s}$.

It is easy to see that we have vertices $u_1,u_2\in W$ such that  
\[
\lambda_1(\mathcal{F}(P_1)_{v_1})+\beta(\mathcal{F}(W)_{u_1})=0=\lambda_2(\mathcal{F}(P_2)_{v_2})+\beta(\mathcal{F}(W)_{u_2}).
\]
So we can form the connected sum
\[
(P_1,\lambda_1)\#_{v_1,u_1}(W,\beta)\#_{u_2,v_2}(P_2,\lambda_2).
\]
This has quasitoric polynomial equal to
\[
\mathfrak{g}(P_1,\lambda_1)+\mathfrak{g}(W,\beta)+\mathfrak{g}(P_2,\lambda_2)=h_1+\mathfrak{g}(W,\beta)+h_2
\]
by Lemma~\ref{L:connsum}. By~\eqref{E:Q} and~\eqref{E:Q1}, the quasitoric polynomial $\mathfrak{g}(W,\beta)=0$ so we have that $h_1+h_2$ is the quasitoric polynomial of a quasitoric pair.

We have therefore shown that if we have a monomial of $h_1$ and a monomial of $h_2$, that differ by a single indeterminate, we can find a quasitoric pair whose quasitoric polynomial is $h_1+h_2$. We now consider the situation where there are no monomials of $h_1$ that have any indeterminates in common with any monomial of $h_2$. Our course of action will be to take a monomial of $h_1$ and perform the connected sum operation of $(P_1,\lambda_1)$ with $n$ bounding pairs such that, after each connected sum, there will be a monomial of the resulting quasitoric polynomial that has one more indeterminate in common with a monomial of $h_2$. 

Suppose $h_1$ contains a monomial $s_1\wedge s_2\wedge\dots\wedge s_n$ and $h_2$ contains a monomial $-\widetilde{s_1}\wedge \widetilde{s_2}\wedge\dots\wedge \widetilde{s_n}$, where $s_i\ne \widetilde{s_j}$ for any $1\leq i,j\leq n$. From what we have shown above we can obtain a quasitoric pair $(P_1,\lambda_1)\#_{v_1,u_1}(W,\beta_1)$ such that there is a vertex $p_1$ with associated monomial $\widetilde{s_1}\wedge s_2\wedge\dots\wedge s_n$. In the same way as above we construct quasitoric pairs $(Q,\eta_2)$ and $(Q,\widetilde{\eta_2})$ such that
\[
\eta_2(F)=
\begin{cases}
\widetilde{s_2}, &F=\{0\}\times \Delta^{n-1};\\
s_2, &F=\{1\}\times \Delta^{n-1};\\
\widetilde{s_1}, &F=\Delta^1\times D_1;\\
s_{i+1}, &F=\Delta^1\times D_i,\ \text{for}\ 2\leq i<n;\\
\mathfrak{s}_2, &F=\Delta^1\times D_n,
\end{cases}
\]
where $\mathfrak{s}_2\letbe \widetilde{s_1}+s_3+\dots+s_n\in \Z^n$. Similarly we define
\[
\widetilde{\eta_2}(F)=
\begin{cases}
s_2, &F=\{0\}\times \Delta^{n-1};\\
\widetilde{s_2}, &F=\{1\}\times \Delta^{n-1};\\
\widetilde{s_1}, &F=\Delta^1\times D_1;\\
s_{i+1}, &F=\Delta^1\times D_i,\ \text{for}\ 2\leq i<n;\\
\mathfrak{s}_2, &F=\Delta^1\times D_n.
\end{cases}
\]
Obviously,
\[
\mathfrak{g}(Q,\eta_2)+\mathfrak{g}(Q,\widetilde{\eta_2})=0
\]
and we define
\[
(W,\beta_2)\letbe (Q,\eta_2)\#_{q_2,\widetilde{q_2}} (Q,\widetilde{\eta_2})
\]
taking the connect sum at vertices that have associated monomials containing the indeterminates $s_2,s_3,\dots,s_n,\mathfrak{s}_2$.

The pair $(W,\beta_2)$ has a vertex $p_1'$ with associated monomial $-\widetilde{s_1}\wedge s_2\wedge\dots\wedge s_n$. We can therefore take the connect sum
\[
(P_1,\lambda_1)\#_{v_1,u_1}(W,\beta_1)\#_{p_1,p_1'} (W,\beta_2),
\]
which has a vertex $p_2$ with associated monomial $\widetilde{s_1}\wedge \widetilde{s_2}\wedge s_3\wedge\dots\wedge s_n$.

Continuing this procedure, we can further construct a series of quasitoric pairs $(W,\beta_3),\dots,(W,\beta_n)$, all with zero quasitoric polynomial, so that we can finally obtain a quasitoric pair
\begin{align*}
(P_1,\lambda_1)\#_{v_1,u_1}(W,\beta_1)\#_{p_1,p_1'} (W,\beta_2)\#_{p_2,p_2'}\cdots\#_{p_{n-1},p_{n-1}'} (W,\beta_n),
\end{align*}
such that there is a vertex $p_n$ with associated monomial $\widetilde{s_1}\wedge\widetilde{s_2}\wedge\dots\wedge \widetilde{s_n}$. We can therefore perform the connected sum operation of this quasitoric pair with $(P_2,\lambda_2)$ to obtain a quasitoric pair $(P,\lambda)$ such that
\[
\mathfrak{g}(P,\lambda)=\mathfrak{g}(P_1,\lambda_1)+\mathfrak{g}(P_2,\lambda_2)=h_1+h_2.
\]
\end{proof}

\subsubsection*{The 2-dimensional Case}

A connected polynomial $h\in \mathfrak{K}_2$ is of the form
\[
h=s_1\wedge s_2+s_2\wedge s_3+\dots+s_{n-1}\wedge s_n+s_n\wedge s_1,
\]
where $s_i\in J_n^*$, for $1\leq i\leq n$, and $s_i\neq s_j$, for $i\neq j$. Note that for $h$ to be non-zero, $n$ must be greater than~2. If we take an $n$-sided polygon $P\subset \R^2$ and label its facets, in a clockwise direction, by $s_1,\dots,s_n$, then we have a quasitoric pair whose quasitoric polynomial is $h$.

Therefore, if we take any non-zero polynomial $h\in \mathfrak{K}$ and decompose it into connected parts $h=h_1+\dots+h_s$, for each $h_i$ we have a quasitoric pair $(P_i,\lambda_i)$ and
\[
\mathfrak{g}((P_1,\lambda_1)+\dots+(P_s,\lambda_s))=h.
\]
By Proposition~\ref{P:add}, we know that there is a quasitoric pair $(P,\lambda)$ such that $\mathfrak{g}(P,\lambda)=h$. We therefore have

\begin{corollary}
The homomorphism of abelian groups
\[
\mathcal{M}\colon \mathcal{Q}_2\longrightarrow \Xi_2
\]
is surjective and, in particular, for each $x\in \Xi_2$, there exists a quasitoric pair $(P,\lambda)$ such that $\mathcal{M}(P,\lambda)=x$. 
\end{corollary}

\subsubsection*{The 3-dimensional Case}

The boundary of a simple $n$-polytope is homeomorphic to an $(n-1)$-sphere. Since connected closed 1-dimensional manifolds are 1-spheres we have no problem in the lower dimensional cases.

Take a simplicial triangulation $K$ of a closed surface that is not a sphere, a torus say, and choose a vector in $\Z^3$ for each vertex in such a way that the three vectors associated to the vertices of any 2-simplex form a basis. We can then define an exterior polynomial $h$ by consistently orienting the simplicies and taking the exterior product of the vectors that form a basis of a 2-simplex and summing over all the 2-simplicies. Then $h$ will be faithful and will satisfy $d(h)=0$, so $h$ will be an element of $\mathfrak{K}_3$. For $\mathcal{M}\colon \mathcal{Q}_3\to \Xi_3$ to be surjective we will need that $h$ is the quasitoric polynomial of a quasitoric pair $(P,\lambda)$; but the geometric realisation of $K$ is not a sphere. This can be resolved in certain cases as shown in the following example.

\begin{example}
Consider the following triangulation of the torus where the outer edges are identified in the normal fashion:
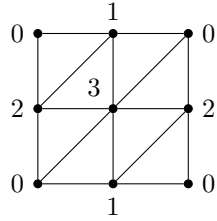
\[ \begin{tikzpicture}
\vertex (v1) at (0,0) [label=left:$0$] {};
\vertex (v2) at (1,0) [label=below:$1$] {};
\vertex (v3) at (2,0) [label=right:$0$] {};
\vertex (v4) at (0,1) [label=left:$2$] {};
\vertex (v5) at (1,1) [label=above left:$3$] {};
\vertex (v6) at (2,1) [label=right:$2$] {};
\vertex (v7) at (0,2) [label=left:$0$] {};
\vertex (v8) at (1,2) [label=above:$1$] {};
\vertex (v9) at (2,2) [label=right:$0$] {};
\path
(v1) edge (v2)
(v1) edge (v5)
(v1) edge (v4)
(v2) edge (v3)
(v2) edge (v5)
(v3) edge (v6)
(v4) edge (v5)
(v4) edge (v7)
(v5) edge (v6)
(v2) edge (v6)
(v4) edge (v8)
(v7) edge (v8)
(v5) edge (v8)
(v5) edge (v9)
(v6) edge (v9)
(v8) edge (v9)
;
\end{tikzpicture} \]
We associate to each vertex an element $s_i\in J_3^*$, for $0\leq i\leq 3$, with $s_i\ne s_j$ for $i\neq j$, such that three vertices corresponding to a 2-simplex form a basis of $\Z^3$. Orienting all the 2-simplicies clockwise gives us the following polynomial associated to this simplicial triangulation
\[
h=2(s_0\wedge s_1\wedge s_2+s_1\wedge s_3\wedge s_2+ s_0\wedge s_3\wedge s_1+s_0\wedge s_2\wedge s_3)\in\mathfrak{K}_3.
\]
This can be seen to be the quasitoric polynomial of two disjoint copies of $(\Delta^3,\delta)$, where $\delta(D_i)=s_i$, for $0\leq i\leq 3$. By Proposition~\ref{P:add} we then know that there is a quasitoric pair $(P,\lambda)$ such that $\mathfrak{g}(P,\lambda)=h$.
\end{example}

We now present the following similar example for which we do not have a solution in general as of yet.

\begin{example}
Consider the following triangulation of the torus:
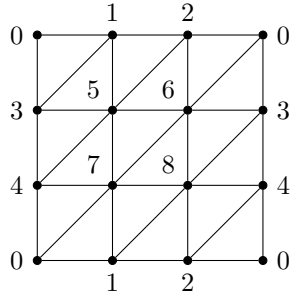
\[ \begin{tikzpicture}
\vertex (v1) at (0,0) [label=left:$0$] {};
\vertex (v2) at (1,0) [label=below:$1$] {};
\vertex (v3) at (2,0) [label=below:$2$] {};
\vertex (v4) at (0,1) [label=left:$4$] {};
\vertex (v5) at (1,1) [label=above left:$7$] {};
\vertex (v6) at (2,1) [label=above left:$8$] {};
\vertex (v7) at (0,2) [label=left:$3$] {};
\vertex (v8) at (1,2) [label=above left:$5$] {};
\vertex (v9) at (2,2) [label=above left:$6$] {};
\vertex (v10) at (3,0) [label=right:$0$] {};
\vertex (v11) at (3,1) [label=right:$4$] {};
\vertex (v12) at (3,2) [label=right:$3$] {};
\vertex (v13) at (3,3) [label=right:$0$] {};
\vertex (v14) at (2,3) [label=above:$2$] {};
\vertex (v15) at (1,3) [label=above:$1$] {};
\vertex (v16) at (0,3) [label=left:$0$] {};
\path
(v1) edge (v2)
(v1) edge (v5)
(v1) edge (v4)
(v2) edge (v3)
(v2) edge (v5)
(v3) edge (v6)
(v4) edge (v5)
(v4) edge (v7)
(v5) edge (v6)
(v2) edge (v6)
(v4) edge (v8)
(v7) edge (v8)
(v5) edge (v8)
(v5) edge (v9)
(v6) edge (v9)
(v8) edge (v9)
(v3) edge (v10)
(v3) edge (v11)
(v10) edge (v11)
(v6) edge (v11)
(v11) edge (v12)
(v6) edge (v12)
(v9) edge (v12)
(v12) edge (v13)
(v13) edge (v9)
(v14) edge (v9)
(v13) edge (v14)
(v8) edge (v14)
(v8) edge (v15)
(v14) edge (v15)
(v7) edge (v15)
(v7) edge (v16)
(v15) edge (v16)
;
\end{tikzpicture} \]
The associated polynomial can be written as
\begin{align*}
h=&s_0\wedge(s_1\wedge s_3+s_3\wedge s_6+s_6\wedge s_2+s_2\wedge s_4+s_4\wedge s_7+s_7\wedge s_1)\\
+&s_5\wedge(s_1\wedge s_2+s_2\wedge s_6+s_6\wedge s_7+s_7\wedge s_4+s_4\wedge s_3+s_3\wedge s_1)\\
+&s_8\wedge(s_1\wedge s_7+s_7\wedge s_6+s_6\wedge s_3+s_3\wedge s_4+s_4\wedge s_2+s_2\wedge s_1)\in \mathfrak{K}_3.
\end{align*}
The polynomials inside the brackets can be seen to be associated to hexagons $P_6$ as in the 2-dimensional case. We take three pairs $(P_6\times I,\lambda_i)$, for $i=1,2,3$, where the labelling on the facets of the hexagonal prisms of the form $F\times I$ is induced by the labelling on the hexagons and the labelling on the end facets is given by
\[
\lambda_i(P_6\times \{0\})=
\begin{cases}
s_0, &i=1;\\
s_5, &i=2;\\
s_8, &i=3,
\end{cases}
\]
and
\[
\lambda_i(P_6\times \{1\})=t, \quad \text{for i=1,2,3},
\]
for some $t\in J_3^*$. The sum of the polynomials associated to these three pairs is then equal to $h$. The problem is that these three pairs are not necessarily quasitoric pairs as they do not necessarily satisfy Condition~($\star$). We need to find a $t$ such that at each of the three ends $P_6\times \{ 1\}$ we have a basis associated to each vertex. We do not have a proof yet that a $t$ exists in general but all attempts at finding an example where a $t$ does not exist have failed.
\end{example}

We end with the following conjecture:

\begin{conjecture}
The homomorphism
\[
\mathcal{M}\colon \mathcal{Q}_*\longrightarrow \Xi_*
\]
is surjective.
\end{conjecture}


\begin{thebibliography}{99}


\bibitem{BP02}
Victor M. Buchstaber and Taras E. Panov.
\emph{Torus Actions and Their Applications in Topology and Combinatorics}.
University Lecture Series 24. AMS, Providence RI, 2002.

\bibitem{BR98}
Victor M. Buchstaber and Nigel Ray.
Toric Manifolds and Complex Cobordisms.
\emph{Russian Mathematical Surveys} 53(2):371--412, 1998.

\bibitem{BR01}
Victor M. Buchstaber and Nigel Ray.
Tangential Structures on Toric Manifolds, and Connected Sums of Polytopes.
\emph{International Mathematics Research Notices} 4:193--219, 2001.

\bibitem{BPR07}
Victor M. Buchstaber, Taras E. Panov and Nigel Ray.
Spaces of Polytopes and Cobordism of Quasitoric Manifolds.
\emph{Moscow Mathematical Journal} 7(2):219--242, 2007.

\bibitem{BPR10}
Victor M. Buchstaber, Taras E. Panov and Nigel Ray.
Toric Genera.
\emph{International Mathematics Research Notices} 16:3207--3262, 2010.

\bibitem{Com96}
Gustavo Comeza\~na.
Calculations in Complex Equivariant Bordism.
Chapter~XXVIII of~\cite{May96}:333--352, 1996.

\bibitem{CF64}
Pierre E. Conner and Edwin E. Floyd.
\emph{Differentiable Periodic Maps}.
Ergebnisse der Mathematik 33. Springer-Verlag, 1964.

\bibitem{CGK02}
Michael Cole, John P. C. Greenlees and Igor Kriz.
The Universality of Equivariant Complex Bordism.
\emph{Mathematische Zeitschrift} 239(3):455--475, 2002.

\bibitem{D13}
Alastair Darby.
\emph{Quasitoric Manifolds in Equivariant Complex Bordism.}
Doctoral Thesis. The University of Manchester, 2013.

\bibitem{tD70}
Tammo tom Dieck.
Bordism of $G$-Manifolds and Integrality Theorems.
\emph{Topology} 9(4):345--358, 1970.

\bibitem{DJ91}
Michael W. Davis and Tadeusz Januszkiewicz.
Convex Polytopes, Coxeter Orbifolds and Torus Actions.
\emph{Duke Mathematical Journal} 62(2):417--451, 1991.

\bibitem{GZ99}
Victor Guillemin and Catalin Zara.
Equivariant de Rham Theory and Graphs.
\emph{The Asian Journal of Mathematics} 3(1):49--76, 1999.

\bibitem{GZ01}
Victor Guillemin and Catalin Zara.
$G$-Actions on Graphs.
\emph{International Mathematics Research Notices} 10:519--542, 2001.


\bibitem{GKM98}
Mark Goresky, Robert Kottwitz and Robert MacPherson.
Equivariant Cohomology, Koszul Duality and the Localisation Theorem.
\emph{Inventiones Mathematicae} 131(1):25--83, 1998.

\bibitem{Han05}
Bernhard Hanke.
Geometric Versus Homotopy Theoretic Equivariant Bordism.
\emph{Mathematische Annalen} 332(3):677--696, 2005.

\bibitem{Hat74}
Akio Hattori.
Equivariant Characteristic Numbers and Integrality Theorem for Unitary $T^n$-Manifolds.
\emph{T\^ohoku Mathematical Journal} 26(3):461--482, 1974.

\bibitem{HO72}
Gary Hamrick and Erich Ossa.
Unitary Bordism of Monogenic Groups and Isometries.
\emph{Proceedingds of the Second Conference on Transformation Groups.}
Lecture Notes in Mathematics 298:172--182, 1972.

\bibitem{K07}
Marja Kankaanrinta.
Equivariant Collaring, Tubular Neighbourhood and Gluing Theorems for Proper Lie Group Actions.
\emph{Algebraic and Geometric Topology} 7:1--27, 2007.

\bibitem{L73}
Peter L\"offler.
Characteristic Numbers of Unitary Torus-Manifolds.
\emph{Bulletin of the American Mathematical Society} 79(6):1262--1263, 1973.

\bibitem{LT13}
Zhi L\"u and Qiangbo Tan.
Small Covers and the Equivariant Bordism Classification of 2-torus Manifolds.
\emph{International Mathematics Research Notices} doi:10.1093/imrn/rnt183, 2013.

\bibitem{Mas99}
Mikiya Masuda.
Unitary Toric Manifolds, Multi-Fans and Equivariant Index.
\emph{Tohoku Mathematical Journal} 51(2):237--265, 1999.

\bibitem{May96}
J. Peter May et al.
\emph{Equivariant Homotopy and Cohomology Theory}.
CBMS Regional Conference Series in Mathematics, Volume 91. AMS, 1996.

\bibitem{MMP07}
Hiroshi Maeda, Mikiya Masuda and Taras Panov.
Torus Graphs and Simplicial Posets.
\emph{Advances in Mathematics} 212(2):458--483, 2007.

\bibitem{Q71}
Daniel G. Quillen.
Elementary Proofs of Some Results of Cobordism Theory Using Steenrod Operations.
\emph{Advances in Mathematics} 7(1):29--56, 1971.

\bibitem{R03}
Ioanid Rosu.
Equivariant $K$-Theory and Equivariant Cohomology.
\emph{Mathematische Zeitschrift} 243(3):423-448, 2003.

\bibitem{Seg68}
Graeme Segal.
Equivariant $K$-Theory.
\emph{Publications Math\'ematiques de l'H\'ES} 34(1):129--151, 1968.

\bibitem{Sin01}
Dev P. Sinha.
Computations of Complex Equivariant Bordism Rings.
\emph{American Journal of Mathematics} 123(4):577--605, 2001.

\bibitem{Z95}
G\"unter M. Ziegler.
\emph{Lectures on Polytopes}.
Graduate Texts in Mathematics~152. Springer-Verlag, 1995.

\end{thebibliography}
\end{document}